\documentclass[reqno]{amsart}
\usepackage[utf8]{inputenc}
\usepackage{latexsym,amssymb,amsmath,amsthm,amscd,graphicx,enumerate}
\usepackage{amsfonts}
\usepackage{xcolor}
\usepackage{tikz}
\usepackage[english]{babel}
\usepackage[a4paper,bindingoffset=0.2in,%
left=1in,right=1in,top=1in,bottom=1in,%
footskip=.25in]{geometry}

\newtheorem{theorem}{Theorem}[section]
\newtheorem{lemma}[theorem]{Lemma}
\newtheorem{prop}[theorem]{Proposition}
\newtheorem{cor}[theorem]{Corollary}
\theoremstyle{definition}
\newtheorem{defn}[theorem]{Definition}

\newtheorem{remark}[theorem]{Remark}
\numberwithin{equation}{section}

\DeclareMathOperator*{\essinf}{ess\, inf}

\begin{document}
	
\title[Two-weighted estimates for some sublinear operators and applications]{Two-weighted estimates for some sublinear operators on generalized weighted Morrey Spaces and applications}
	
\author[Y.~Ramadana]{Yusuf Ramadana}
\address{Faculty of Mathematics and Natural Sciences, State University of Makassar, Makassar 90221, Indonesia}
\email{yusuframadana.96@gmail.com}
	
\author[H.~Gunawan]{Hendra Gunawan}
\address{Faculty of Mathematics and Natural Sciences, Bandung Institute of Technology, Bandung 40132, Indonesia}
\email{hgunawan@itb.ac.id}
	
	
	
\begin{abstract}
In this paper we investigate the boundedness of sublinear operators generated by fractional 
integrals as well as sublinear operators generated by Calder\`on-Zygmund operators on generalized weighted Morrey spaces 
and generalized weighted mixed-Morrey spaces. In particular we are interested in the 
strong-type estimate for $1<p<\infty$ and the weak estimate for $p=1$. Under some 
assumptions, we prove that the operators and their commutator with a BMO function are 
bounded on those function spaces with different weights. The results then imply the 
boundedness of fractional integrals with Gaussian kernel bounds as well as with rough 
kernels, fractional maximal integrals with rough kernels, and sublinear operators with rough 
kernels generated by Calder\`on-Zygmund operators. Using the results, we obtain some 
regularity properties of the solution of some partial differential equations.
\end{abstract}
	
\maketitle
	
\section{Introduction}

We shall discuss the boundedness of sublinear operators on generalized weighted Morrey 
spaces and generalized weighted mixed-Morrey spaces with different weights and then apply 
the results to the regularity properties of the solution of some elliptical differential 
equations.

The classical Morrey spaces $L^{p,\lambda}$ were first introduced by C.B. Morrey in 1938 
\cite{Morrey}. After the WW II, a large number of investigations regarding the functions 
spaces have been carried out. The reader may find some recent works in 
\cite{Aykol, Guliyev2014, Ho, Kucukaslan, Ramadana} and the references therein.

Throughout this paper, we denote by $B(a,r)$ an open ball centered at $a\in\mathbb{R}^n$ 
with radius $r>0$. For a set $E$ in $\mathbb{R}^n$, we denote by $E^{\rm c}$ the complement 
of $E$. Moreover, if $E$ is a measurable set in $\mathbb{R}^n$, then $|E|$ denotes the 
Lebesgue measure of $E$.

Let $0 \le \alpha <n, 1 \leq  p < n/\alpha,$ and $1/q = 1/p - \alpha/n.$ We consider a 
sublinear operator $S_\alpha$ which satisfies the inequality
\begin{equation} \label{sublinear}
    \sup_{x \in B(a,r)} \left|S_\alpha \left( f \cdot \mathcal{X}_{B(a,2r)^c} \right)(x)\right| \leq
C \sum_{k=1}^\infty (2^{k+1}r)^{-\frac{n}{s}+\alpha} \left(\int_{B(a,2^{k+1}r)} |f(y)|^{s} dy \right)^{1/s},
\end{equation}
for any $(a,r) \in \mathbb{R}^n \times \mathbb{R}^+$, where $1<s<p$ for $1<p<n/\alpha$
and $s=1$ for $p=1$. This operator was introduced by H.~Wang 
\cite{Wang}. Moreover, the commutator of the sublinear operator is defined by
$$
[b,S_\alpha] (f) (x) = (bS_\alpha (f))(x) - S_\alpha (bf) (x), \quad x \in \mathbb{R}^n.
$$
We see that
\begin{equation}\label{commutator}
|S_\alpha (f)(x)| \leq |b(x) - b_B| |S_\alpha(f)(x)| + |S_\alpha ((b-b_B)f) (x)|, 
\quad x \in B,
\end{equation}
for any ball $B$ in $\mathbb{R}^n.$ Assuming that this operator and its commutator are 
bounded on weighted Lebesgue spaces (or weighted weak Lebesgue spaces), Wang obtained the 
following theorem on their boundedness property on generalized weighted Morrey spaces (or 
generalized weighted weak Morrey spaces).

\begin{theorem} \cite{Wang}
Let $0<\alpha<n,$ $1\le p<n/\alpha, 1/q = 1/p - \alpha/n,$ $w^{s} \in A_{p/s,q/s}$ where
$1\le s<p$ for $1<p<n/\alpha$ and $s=1$ for $p=1$, and 
$S_\alpha$ be a sublinear operator satisying (\ref{sublinear}) for any $(a,r) \in 
\mathbb{R}^n \in \mathbb{R}^+.$
\begin{enumerate}
\item Suppose that $(\psi_1, \psi_2)$ satisfies the condition
$$
\int_r^\infty \frac{\essinf_{t<l<\infty} \psi_1(a,l) w^p(B(a,l))^\frac1p}
{w^p(B(a,t))^\frac1p} \frac{dt}t \leq C \psi_2(a,r), \quad (a,r) 
\in \mathbb{R}^n \times \mathbb{R}^+.
$$
If $S_\alpha$ is bounded from $L^{p,w^p}$ to $L^{q,w^q}$ for $1<p<n/\alpha$ and from 
$L^{1,w}$ to $WL^{1,w}$, then $S_\alpha$ is bounded from $\mathcal{M}^{p,w^p}_{\psi_1}$ 
to $\mathcal{M}^{q,w^q}_{\psi_2}$ for $1<p<n/\alpha$ and from $\mathcal{M}^{1,w}_{\psi_1}$ 
to $W\mathcal{M}^{q,w^q}_{\psi_2}$.
\item Suppose that $(\psi_1, \psi_2)$ satisfies the condition
$$
\int_r^\infty \left(1 + \ln \frac{r}{t}\right) \frac{\essinf_{t<l<\infty} \psi_1(a,l) 
w^p(B(a,l))^\frac1p}{w^p(B(a,t))^\frac1p} \frac{dt}t \leq C \psi_2(a,r), \quad (a,r) 
\in \mathbb{R}^n \times \mathbb{R}^+.
$$
If $b \in BMO$ and $[b, S_\alpha]$ is bounded from $L^{p,w^p}$ to $L^{q,w^q}$ for 
$1<p<n/\alpha$ and from $L^{1,w}$ to $WL^{1,w}$, then $[b,S_\alpha]$ is bounded from 
$\mathcal{M}^{p,w^p}_{\psi_1}$ to $\mathcal{M}^{q,w^q}_{\psi_2}$ for $1<p<n/\alpha$ and 
from $\mathcal{M}^{1,w}_{\psi_1}$ to $W\mathcal{M}^{q,w^q}_{\psi_2}$.
\end{enumerate}
\end{theorem}

In this paper, we investigate the boundedness of the sublinear operators on generalized
weighted Morrey spaces and generalized weighted mixed-Morrey spaces with different weights 
(which we shall define in Section 2), where $w^s\in A_{p/s,q/s}$ is replaced by 
$(w_1^s,w_2^s)\in A_{p/s,q/s}$ for $1\le p<n/\alpha$. Moreover, we also investigate 
the operator $S_0$ for the case $\alpha=0$ to extend the known results for the sublinear 
operators generated by Calder\`on-Zygmund operators on generalized weighted Morrey spaces 
and generalized weighted mixed-Morrey spaces. Recall that $S$ is a sublinear operator 
generated by Calder\'on-Zygmund operator if $S$ satisfies 
\begin{equation}\label{sublinear-calder-1}
|S(f)(x)| \leq C \int_{\mathbb{R}^n} \frac{|f(y)|}{|x-y|^n} dy
\end{equation}
for $f\in L^1$ with compact support and $x \notin \rm{supp}$ $(f)$. We note such an operator satisfies 
\begin{equation}\label{sublinear-calder}
\sup_{x\in B(a,r)} |S(f\cdot\mathcal{X}_{B(a,2r)^c})(x)|\leq C \sum_{k=1}^\infty (2^{k+1}r)^{-\frac{n}s}\left(\int_{B(a,2^{k+1}r)}|f(y)|^s dy\right)^\frac1s,
\end{equation}
which is the inequality (\ref{sublinear}) for $S_0$. In other words, the inequality 
(\ref{sublinear-calder}) is more general than (\ref{sublinear-calder-1}). Hereafter, 
without confusion, we will write $S:=S_0.$

This paper is quite long. For convenience, we state our main results below. The definitions
and the properties of $A_p$ and $A_{p,q}$ weights will be presented in Section 2, along with 
the definitions of the function spaces that we are working on. The proofs of the theorems 
will be presented in Section 3.

\begin{theorem}\label{Main-Riesz}
Let $0 < \alpha < n, 1 \le p < n/\alpha, 1/q = 1/p - \alpha/n$, $(w_1^s, w_2^s) \in 
A_{p/s,q/s}$ and $w_2^s \in A_{p/s, q/s}$ where $1\le s<p$ for $1<p<n/\alpha$ and $s=1$ for 
$p=1$. Suppose that the two positive functions $\psi_1$ 
and $\psi_2$ on $\mathbb{R}^n \times \mathbb{R}^+$ satisfy
\begin{equation}\label{inequal-Riesz}
\sup_{(a,r) \in \mathbb{R}^n \times \mathbb{R}^+} \frac{1}{\psi_2 (a,r)}\int_r^\infty 
\frac{\inf_{t\leq l<\infty} \psi_1(a,l)w_1^p(B(a,l))^\frac1p}{w_2^q(B(a,t))^\frac1q} 
\frac{dt}{t} \leq C < \infty.
\end{equation}
If $S_\alpha$ is bounded from $L^{p,w_1^p}$ to $L^{q,w_2^q}$ for $1<p<n/\alpha$ and 
from $L^{1,w_1}$ to $WL^{q,w_2^q}$, then $S_\alpha $ is bounded from 
$\mathcal{M}^{p,w_1^p}_{\psi_1}$ to $\mathcal{M}^{q,w_2^q}_{\psi_2}$ for $1<p<n/\alpha$ 
and from $\mathcal{M}^{1,w_1}_{\psi_1}$ to $W\mathcal{M}^{q,w_2^q}_{\psi_2}$,
\end{theorem}

\begin{theorem}\label{Main-Comm-Riesz}
Let $0 < \alpha < n, 1 < p < n/\alpha, 1/q = 1/p - \alpha/n$, $(w_1^s, w_2^s) \in 
A_{p/s,q/s}$ and $w_2^s \in A_{p/s, q/s}$ where $1\le s<p$, and $w_1 \in A_\infty$. Suppose that the two 
positive functions $\psi_1$ and $\psi_2$ on $\mathbb{R}^n \times \mathbb{R}^+$ satisfy
\begin{equation}\label{inequal-comm-riesz}
\sup_{(a,r) \in \mathbb{R}^n \times \mathbb{R}^+} \frac{1}{\psi_2 (a,r)}\int_r^\infty 
\left(1+\ln \frac{t}{r}\right) \frac{\inf_{t\leq l<\infty} \psi_1(a,l)w_1^p(B(a,l))^\frac1p}
{w_2^q(B(a,t))^\frac1q} \frac{dt}{t} \leq C < \infty.
\end{equation}
If $[b, S_\alpha]$ is bounded from $L^{p,w_1^p}$ to $L^{q,w_2^q}$, then 
$[b, S_\alpha]$ is bounded from $\mathcal{M}^{p,w_1^p}_{\psi_1}$ to 
$\mathcal{M}^{q,w_2^q}_{\psi_2}$. Moreover,
$$
\|[b, S_\alpha]f\|_{\mathcal{M}^{q,w_2^q}_{\psi_1}} \leq C \|b\|_* 
\|f\|_{\mathcal{M}^{p,w_1^p}_{\psi_1}}, \quad f\in \mathcal{M}^{p,w^p}_{\psi_1}.
$$
\end{theorem}

\begin{theorem}\label{Main-Riesz-Mixed}
Let $\psi$ be a positive function on $\mathbb{R}^n\times (0,\infty)$, $1 \leq p_0 < \infty,
\ 0 < \alpha < n,\ 1 \le p < n/\alpha,\ 1/q = 1/p - \alpha/n$, $(w_1^s, w_2^s)\in 
A_{p/s, q/s}$ and $w_2^s \in A_{p/s, q/s}$ where $1\le s<p$ for $1<p<n/\alpha$ and $s=1$ for
$p=1$. Suppose that the two positive functions $\psi_1$
and $\psi_2$ on $\mathbb{R}^n \times \mathbb{R}^+$ satisfy (\ref{inequal-Riesz}).
If $S_\alpha$ is bounded from $L^{p,w_1^p}$ to $L^{q,w_2^q}$ for $1<p<n/\alpha$ and from 
$L^{1,w_1}$ to $WL^{q,w_2^q}$, then $S_\alpha $ is bounded from 
$\mathcal{M}_\psi^{p_0,w}(0,T,\mathcal{M}^{p,w_1^p}_{\psi_1})$ to
$\mathcal{M}_\psi^{p_0,w}(0,T,\mathcal{M}^{q,w_2^q}_{\psi_2})$ for $1<p<n/\alpha$ and from
$\mathcal{M}_\psi^{p_0,w}(0,T,\mathcal{M}^{1,w_1}_{\psi_1})$ to 
$\mathcal{M}_\psi^{p_0,w}(0,T,W\mathcal{M}^{q,w_2^q}_{\psi_2})$.
\end{theorem}

\begin{theorem}\label{Main-Comm-Riesz-Mixed}
Let $\psi$ be a positive function on $\mathbb{R}^n\times (0,\infty)$, $1 \leq p_0 < \infty,\ 
0 < \alpha < n,\ 1 < p < n/\alpha,\ 1/q = 1/p - \alpha/n$, $(w_1^s, w_2^s) \in 
A_{p/s,q/s}$ and $w_2^s \in A_{p/s, q/s}$ where $1\le s<p$, and $w_1 \in A_\infty$. 
Suppose that the two positive functions $\psi_1$ and $\psi_2$ on $\mathbb{R}^n \times 
\mathbb{R}^+$ satisfy (\ref{inequal-comm-riesz}).
If $[b, S_\alpha]$ is bounded from $L^{p,w_1^p}$ to $L^{q,w_2^q}$, then $[b, S_\alpha]$ is 
bounded from $\mathcal{M}_\psi^{p_0,w}(0,T,\mathcal{M}^{p,w_1^p}_{\psi_1})$ to 
$\mathcal{M}_\psi^{p_0, w}(0,T,\mathcal{M}^{q,w_2^q}_{\psi_2})$. Moreover,
$$
\|[b, S_\alpha]f\|_{\mathcal{M}_\psi^{p_0, w}(0, T, \mathcal{M}^{q,w_2^q}_{\psi_1})} \leq 
C \|b\|_*\|f\|_{\mathcal{M}_\psi^{p_0, w}(0, T, \mathcal{M}^{p,w_1^p}_{\psi_1})},\quad f\in
\mathcal{M}_\psi^{p_0, w}(0, T, \mathcal{M}^{p,w_1^p}_{\psi_1}).
$$
\end{theorem}

For the case $\alpha=0$, we have the following results for the operator $S=S_0$.

\begin{theorem}\label{Main-Calder}
Let $1 \le p < \infty$, $(w_1^s,w_2^s)\in A_{p/s}$ and $w_2^s \in A_{p/s}$ where $1\le s<p$
for $1<p<\infty$ and $s=1$ for $p=1$. Suppose that the two positive functions $\psi_1$ and 
$\psi_2$ on $\mathbb{R}^n \times \mathbb{R}^+$ satisfy
\begin{equation}\label{inequal-calder}
\sup_{(a,r) \in \mathbb{R}^n \times \mathbb{R}^+} \frac{1}{\psi_2 (a,r)}\int_r^\infty 
\frac{\inf_{t\leq l<\infty}\psi_1(a,l)w_1(B(a,l))^\frac1p}{w_2(B(a,t))^\frac1p} \frac{dt}{t} 
\leq C < \infty.
\end{equation}
If $S$ is bounded from $L^{p,w_1}$ to $L^{p,w_2}$ for $1<p<\infty$ and from $L^{1,w_1}$ to 
$WL^{1,w_2}$, then $S$ is bounded from $\mathcal{M}^{p,w_1}_{\psi_1}$ to 
$\mathcal{M}^{p,w_2}_{\psi_2}$ for $1<p<\infty$ and from $\mathcal{M}^{1,w_1}_{\psi_1}$
to $W\mathcal{M}^{1,w_2}_{\psi_2}$.
\end{theorem}

\begin{theorem}\label{Main-Comm-Calder}
Let $1 < p<\infty$, $(w_1^s,w_2^s) \in A_{p/s}$ and $w_2^s \in A_{p/s}$ where $1\le s<p$, 
and $w_1\in A_\infty$. Suppose that the two positive 
functions $\psi_1$ and $\psi_2$ on $\mathbb{R}^n \times \mathbb{R}^+$ satisfy
\begin{equation}\label{inequal-comm-calder}
\sup_{(a,r) \in \mathbb{R}^n \times \mathbb{R}^+} \frac{1}{\psi_2 (a,r)}\int_r^\infty 
\left(1+\ln \frac{t}{r}\right) \frac{\inf_{t\leq l<\infty} \psi_1(a,l)w_1(B(a,l))^\frac1p}
{w_2(B(a,l))^\frac1p} \frac{dt}{t} \leq C < \infty.
\end{equation}
If $[b, S]$ is bounded from $L^{p,w_1}$ to $L^{p,w_2}$, then $[b, S]$ is bounded from 
$\mathcal{M}^{p,w_1}_{\psi_1}$ to $\mathcal{M}^{p,w_2}_{\psi_2}$. Moreover,
$$
\|[b, S]f\|_{\mathcal{M}^{p,w_2}_{\psi_1}} \leq C \|b\|_* \|f\|_{\mathcal{M}^{p,w_1}_{\psi_1}}, \quad f\in \mathcal{M}^{p,w^p}_{\psi_1}.
$$
\end{theorem}

\begin{theorem}\label{Main-Calder-Weak}
Let $(w_1,w_2)\in A_1$, $w_2 \in A_1$ and $w_1 \in A_\infty$. Suppose that the two positive 
functions $\psi_1$ and $\psi_2$ on $\mathbb{R}^n \times \mathbb{R}^+$ satisfy
\begin{equation}\label{weak-inequal-comm-calder}
\sup_{(a,r) \in \mathbb{R}^n \times \mathbb{R}^+} \frac{1}{\psi_2 (a,r)}\int_r^{\infty} 
\left(1+\ln \frac{t}r \right)\frac{w_1(B(a,t))}{w_2(B(a,t))}\inf_{t\leq s<\infty} 
\psi_1(a,s) \frac{dt}{t} \leq C < \infty.
\end{equation}
If $\Phi(t) = t (1+\log^+ t)$ and $S$ satisfies
\begin{equation}\label{weak-commutator-calderon}
w_2(\{x \in \mathbb{R}^n: |[b,S](f)(x)|>\sigma\}) \leq C \int_{\mathbb{R}^n} \Phi \left(\frac{|f(x)|}{\sigma}\right) w_1 (x) dx,
\end{equation}
where $C>0$ is independent of $f$ and $\sigma$, then there exists a constant $C>0$ such that
$$
\|[b,S](f)\|_{W\mathcal{M}_{\psi_2}^{1,w_2}} \leq C \|b\|_* \sup_{\sigma>0} \sigma 
\left\|\Phi \left(\frac{|f|}{\sigma}\right)\right\|_{\mathcal{M}_{\psi_1}^{L \log L, w_1}}.
$$
\end{theorem}

\begin{remark}
If $w_1 = w_2 \in A_1$, $b\in BMO$, then the generalized Calder\`on-Zygmund operator 
$T_\theta$ which is introduced by Yabuta \cite{Yabuta} satisfies 
(\ref{weak-commutator-calderon}).
\end{remark}

\begin{theorem}\label{Main-Calder-Mixed}
Let $\psi$ be a positive function on $\mathbb{R}^n \times \mathbb{R}^+$, $1\le p_0<\infty$, 
$1 \leq  p < \infty$, $(w_1^s, w_2^s)\in A_{p/s}$ and $w_2^s \in A_{p/s}$ where $1\le s<p$ 
for $1<p<\infty$ and $s=1$ for $p=1$. Suppose that the two positive functions $\psi_1$ and 
$\psi_2$ on $\mathbb{R}^n \times \mathbb{R}^+$ satisfy (\ref{inequal-calder}).
If $S$ is bounded from $L^{p,w_1}$ to $L^{q,w_2}$ for $1<p<\infty$ and from $L^{1,w_1}$ to 
$WL^{1,w_2}$, then $S$ is bounded from 
$\mathcal{M}_\psi^{p_0,w}(0,T,\mathcal{M}^{p,w_1}_{\psi_1})$ to
$\mathcal{M}_\psi^{p_0,w}(0,T,\mathcal{M}^{p,w_2}_{\psi_2})$ for $1<p<\infty$ and from 
$\mathcal{M}_\psi^{p_0,w}(0,T,\mathcal{M}^{1,w_1}_{\psi_1})$ to 
$\mathcal{M}_\psi^{p_0,w}(0,T,W\mathcal{M}^{1,w_2}_{\psi_2})$.
\end{theorem}

\begin{theorem}\label{Main-Comm-Calder-Mixed}
Let $\psi$ be a positive function on $\mathbb{R}^n \times \mathbb{R}^+$, $1\le p_0<\infty$, 
$1 \leq  p<\infty$, $(w_1^s,w_2^s)\in A_{p/s}$ and $w_2^s \in A_{p/s}$ where $1<s<p$ for
$1<p<\infty$ and $s=1$ for $p=1$, and $w_1 \in A_\infty$. 
Suppose that the two positive functions $\psi_1$ and $\psi_2$ on $\mathbb{R}^n \times 
\mathbb{R}^+$ satisfy (\ref{inequal-comm-calder}).
If $[b, S]$ is bounded from $L^{p,w_1}$ to $L^{p,w_2}$, then $[b, S]$ is bounded from
$\mathcal{M}_\psi^{p_0,w}(0,T,\mathcal{M}^{p,w_1}_{\psi_1})$ to 
$\mathcal{M}_\psi^{p_0,w}(0,T,\mathcal{M}^{q,w_2}_{\psi_2})$. Moreover,
$$
\|[b, S]f\|_{\mathcal{M}_\psi^{p_0,w}(0, T, \mathcal{M}^{p,w_2}_{\psi_2})} \leq C \|b\|_*
\|f\|_{\mathcal{M}_\psi^{p_0,w}(0, T, \mathcal{M}^{p,w_1}_{\psi_1})}, \quad f\in
\mathcal{M}_\psi^{p_0,w}(0,T,\mathcal{M}^{p,w_1}_{\psi_1}).
$$
\end{theorem}

\begin{theorem}\label{Main-Comm-Calder-Mixed-Weak}
Let $\psi$ be a positive function on $\mathbb{R}^n \times \mathbb{R}^+$, $1\le p_0<\infty$, 
$(w_1,w_2)\in A_1$, $w_2 \in A_1$, and $w_1 \in A_\infty$. Suppose that the two positive 
functions $\psi_1$ and $\psi_2$ on $\mathbb{R}^n \times \mathbb{R}^+$ satisfy
(\ref{weak-inequal-comm-calder}).
If $S$ satisfies (\ref{weak-commutator-calderon}), then there exists a constant $C>0$ such 
that for any suitable function $f,$ we have
$$
\|[b,S](f)\|_{\mathcal{M}_\psi^{p_0,w}(0, T, W\mathcal{M}_{\psi_2}^{1,w_2})} \leq C \|b\|_* 
\sup_{\sigma>0} \sigma \left\|\Phi \left(\frac{|f|}{\sigma}\right)
\right\|_{\mathcal{M}_\psi^{p_0,w}(0, T, \mathcal{M}_{\psi_1}^{L \log L, w_1})}.
$$
\end{theorem}

\begin{remark}
Our results generalize the results in \cite{Guliyev2014-1, Guliyev2014, Guliyev2019, Ho, Ramadana, Wang} among others.
\end{remark}

\section{Preliminaries: Definitions and Lemmas}

In this section, we discuss $A_p$ and $A_{p,q}$ weights and the weighted Lebesgue spaces. 
We also present the definitions of generalized weighted Morrey spaces, generalized weighted 
weak Morrey spaces, generalized weighted Morrey spaces of $L \log L$ type, generalized 
weighted mixed-Morrey spaces, generalized weighted mixed-Morrey spaces, generalized weighted 
weak mixed-Morrey spaces, and generalized weighted mixed Morrey spaces of $L \log L$ type. 
At the end of this section, some lemmas that we shall use to prove the main results about 
the boundedness of sublinear operators on generalized weighted Morrey spaces and generalized 
weighted mixed-Morrey spaces are provided.

A weight $w$ is a nonnegative locally integrable function on $\mathbb{R}^n$ taking values 
in the interval $(0,\infty)$ almost everywhere. For a weight $w$ and any measurable set 
$E\subseteq \mathbb{R}^n$, we write $w(E) = \int_E w(x) dx$. The weights that we discuss in 
this paper belongs to the Muckenhoupt class $A_p$ or $A_{p,q}$.

\begin{defn}\cite{Garcia}
For $1 < p < \infty,$ we denote by $A_p$ the set of all weights $w$ on $\mathbb{R}^n$ for 
which there exists a constant $C>0$ such that
$$
\left(\frac{1}{|B(a,r)|} \int_{B(a,r)} w(x) dx\right) \left(\frac{1}{|B(a,r)|} 
\int_{B(a,r)} 
w(x)^{-\frac{1}{p-1}} dx\right)^{p-1}\leq C, \quad (a,r)\in\mathbb{R}^n \times\mathbb{R}^+.
$$
For $p=1$, we denote by $A_1$ the set of all weights $w$ for which there exists a
constant $C>0$ such that
$$
\frac{1}{|B(a,r)|} \int_{B(a,r)} w(x) dx \leq C \|w\|_{L^{\infty}(B(a,r))}, \quad 
(a,r)\in\mathbb{R}^n \times\mathbb{R}^+.
$$
\end{defn}

We define $A_\infty := \cup_{p \geq 1} A_p.$ Note that the last inequality is equivalent to 
$$
\left(\frac{1}{|B(a,r)|} \int_{B(a,r)} w(x) dx)\right) \cdot \|w^{-1}\|_{L^{\infty}
(B(a,r))} \leq C, \quad (a,r)\in\mathbb{R}^n \times\mathbb{R}^+.
$$

\begin{theorem}\label{garcia}\cite{Garcia}
For each $1\leq p < \infty$ and $w \in A_p$, there exists $C>0$ such that $w(B)/w(E) \leq C 
(|B|/|E|)^p$ for every ball $B$ and measurable sets $E \subseteq B.$
\end{theorem}

\begin{defn}
For $1 < p < \infty,$ the pair of weights $(w_1,w_2)$ belongs to $A_p$ if there exists a 
constant $C>0$ such that
$$
\left(\frac{1}{|B(a,r)|} \int_{B(a,r)} w_2(x) dx\right) \left(\frac{1}{|B(a,r)|} 
\int_{B(a,r)} w_1(x)^{-\frac{1}{p-1}} dx\right)^{p-1} \leq C, \quad (a,r)\in\mathbb{R}^n 
\times\mathbb{R}^+.
$$
For $p=1$, the pair of weights $(w_1,w_2)$ belongs to $A_1$ if there exists a positive 
constant $C>0$ such that
$$
\frac{1}{|B(a,r)|} \int_{B(a,r)} w_2(x) dx \leq C \|w_1\|_{L^{\infty}(B(a,r))}, \quad 
(a,r)\in\mathbb{R}^n \times\mathbb{R}^+.
$$
\end{defn}

Associated to a weight $w\in A_p$ with $1 \leq p < \infty$, the weighted Lebesgue space 
$L^{p,w}(\Omega)$ on a measurable subset $\Omega$ is the set of all measurable functions 
$f$ on $\Omega$ for which
$$
\|f\|_{L^{p,w}(\Omega)} := \left(\int_{\Omega} |f(x)|^p w(x) dx\right)^\frac1p < \infty.
$$
In addition, $WL^{p,w}(\Omega)$ consists of all measurable functions $f$ on $\Omega$ for 
which
$$
\|f\|_{WL^{p,w}} := \sup_{\gamma>0} \gamma w(\{x\in \Omega : |f(x)|>\gamma\})^{\frac{1}{p}} 
< \infty.
$$

If $\Omega = \mathbb{R}^n,$ we write $L^{p,w}=L^{p,w}(\mathbb{R}^n)$ and $WL^{p,w}=
WL^{p,w}(\mathbb{R}^n)$. Notice that if $w$ is constant a.e., then we find that $L^{p,w} = 
L^p$ and $WL^{p,w} = WL^p$.
Moreover, we denote $L_{\rm{loc}}^{p,w}$ by the set of any measurable function $f$ such that 
$\|f\cdot\mathcal{X}_{B(a,r)}\|_{L^{p,w}}$ is finite for any ball $B.$

Related to the discussion of fractional integral operators $I_\alpha,$ we have another 
class of weights denoted by $A_{p,q}$.

\begin{defn} \cite{Wheeden, Sawano}
Let $1 < p < q < \infty$ and $p'$ satisfies $1/p+1/p'=1.$ We denote by $A_{p,q}$ the 
collection of all weights $w$ satisfying
$$
\left(\frac{1}{|B(a,r)|} \int_{B(a,r)} w(x)^q dx\right)^{\frac{1}{q}} \left(\frac{1}
{|B(a,r)|} \int_{B(a,r)} w(x)^{-p'} dx \right)^{\frac{1}{p'}} \leq C, \quad (a,r)\in 
\mathbb{R}^n \times (0,\infty)
$$
where $C$ is a constant independent of $a$ and $r.$
For $p=1$ and $q>1$, we denote by $A_{1,q}$ the collection of all weights $w$ for 
which there exists a constant $C>0$ such that for every $(a,r)\in \mathbb{R}^n\times 
(0,\infty)$,
$$
\left(\frac{1}{|B(a,r)|} \int_{B(a,r)} w(x)^q dx\right)^{1/q} \leq C \|w\|_{L^\infty(B(a,r))}
$$
\end{defn}

One may observe that $w \in A_{p,q}$ if and only if $w^q \in A_{q/p'+1}$ for 
$1\le p<q<\infty$ (see \cite{CGG}). Moreover, we have the following lemma.

\begin{lemma}\label{thm2.5} \cite[Theorem 3.2.2]{Lu}
Let $1 \leq p < q<\infty.$ If $w\in A_{p,q}$, then $w^p\in A_p$ and $w^q \in A_q$.
\end{lemma}

\begin{defn} \cite{Aykol}
Let $1 < p < q < \infty$ and $p'$ satisfies $1/p+1/p'=1.$ The pair of weights 
$(w_1,w_2)$ belongs to $A_{p,q}$ if for some $C>0$ it holds that
$$
\left(\frac{1}{|B(a,r)|} \int_{B(a,r)} w_2(x)^q dx\right)^{\frac{1}{q}} \left(\frac{1}
{|B(a,r)|}\int_{B(a,r)} w_1(x)^{-p'} dx \right)^{\frac{1}{p'}} \leq C, \quad (a,r)\in 
\mathbb{R}^n \times \mathbb{R}^+.
$$
For $p=1$, the pair of weights $(w_1,w_2)$ belongs to $A_{1,q}$ if there is a 
constant $C>0$ such that
$$
\left(\frac{1}{|B(a,r)|} \int_{B(a,r)} w_2(x)^q dx\right)^{1/q} \leq C 
\|w_1\|_{L^\infty(B(a,r))}, \quad (a,r)\in \mathbb{R}^n\times  \mathbb{R}^+.
$$
\end{defn}

We rewrite the following results of Muckenhoupt \cite{Muckenhoupt} for the Hardy-Littlewood 
maximal operator $M$, Muckenhoupt and Wheeden \cite{Wheeden} for the Riesz potentials or the
fractional integral operators $I_\alpha$, Coifman and Fefferman \cite{Coifman-Fefferman}, 
Garcia-Cuerva and Rubio de Francia \cite{Garcia}, and E.~Sawyer \cite{Sawyer} for the 
Calder\`on-Zgymund operators $S$ on weighted Lebesgue spaces:

\begin{theorem}\label{thm3.1} \cite{Garcia}
For $1<p<\infty,$ the Hardy-Littlewood maximal operator $M$ is bounded on $L^{p,w}$ if (and
only if) $w\in A_p.$ Moreover, $M$ is bounded from $L^{1,w}$ to $WL^{1,w}$ if $w\in A_1$.
\end{theorem}

\begin{theorem}\label{thm4.1} \cite{Wheeden}
Let $0<\alpha<n, 1<p<n/\alpha, 1/q=1/p-\alpha/n,$ and $w\in A_{p,q}$. Then, the fractional 
integral operator $I_\alpha$ is bounded from $L^{p,w^p}$ to $L^{q,w^q}$. Moreover, if 
$w\in A_{1,p}$ with $1/q = 1-\alpha/n,$ then $I_\alpha$ is bounded from $L^{1,w}$ to 
$WL^{q,w^q}$.
\end{theorem}

\begin{theorem}\label{thm5.1} \cite{Coifman-Fefferman, Garcia, Sawyer}
For $1< p< \infty,$ the Caldero\`on-Zygmund operator $S$ is bounded on $L^{p,w}$ whenever 
$w \in A_p$. Moreover, $S$ is bounded from $L^{1,w}$ to $WL^{1,w}$ whenever $w\in A_1$.
\end{theorem}

The three theorems will imply the boundedness of the three classical operators on 
generalized weighted Morrey spaces. The boundedness properties of these operators shall 
be one of the implications of our results.

Let us now recall some basic definitions and facts for Orlicz functions needed for our 
results, particularly the boundedness for sublinear operators generated by 
Calder\`on-Zygmund operator when $p=1.$ Related to Orlicz functions, we recall the 
definition of Young functions.

\begin{defn}
A function $\Phi:[0,\infty] \to [0,\infty]$ is called a Young function if it is continuous, 
convex, and strictly increasing with $\Phi(0)=0$ and $\lim_{t\to\infty}\Phi(t) = \infty.$
\end{defn}

For a Young function $\Phi$ and $w\in A_\infty$, we define the Orlicz space $L^{\Phi,w}$ to 
be the set of all measurable functions $f$ such that
$$
\int_{\mathbb{R}^n} \Phi(k|f(y)|)w(y) dy < \infty
$$
for some $k>0.$ The space $L^{\Phi,w}$ is a Banach space with respect to the norm
$$
\|f\|_{L^{\Phi,w}} := \inf \left\{ \lambda>0 : \int_{\mathbb{R}^n} \Phi\left(\frac{|f(y)|}
{\lambda}\right) w(y) dy \leq 1 \right\}.
$$
For a measurable subset $\Omega$ in $\mathbb{R}^n$, the space $L^{\Phi,w}(\Omega)$ is 
defined to be the set of all functions $f$ such that
$$
\|f\|_{L^{\Phi,w}(\Omega)} = \inf \left\{ \lambda>0 : \frac{1}{w(\Omega)}\int_{\Omega} 
\Phi\left(\frac{|f(y)|}{\lambda}\right) w(y) dy \leq 1 \right\} <\infty.
$$
Moreover, we denote by $L_{\rm{loc}}^{\Phi,w}$ the set of all measurable functions $f$ such 
that 
$\|f\cdot\mathcal{X}_{B(a,r)}\|_{L^{\Phi,w}}<\infty$ for any ball $B.$

Given a Young function $\Phi$, we denote by $\tilde{\Phi}$ the complementary Young function 
for $\Phi.$ The following inequality holds for any ball $B$ in $\mathbb{R}^n$ (see 
\cite{Rao}).

\begin{equation}\label{GenHolder}
\frac{1}{w(B)} \int_B |f(x)g(x)| w(x) dx \leq 2 \|f\|_{L^{\Phi,w}(B)} 
\|f\|_{L^{\tilde{\Phi},w}(B)}.
\end{equation}

An important example of Young function is $\Phi(t)=t(1+\log ^+ t)$ where $\log^+ t = 
\max (0, \log t).$ Its complementary Young function is $\tilde{\Phi} = e^t -1.$ This 
function will be used in this paper. If $\Phi(t)=t(1+\log ^+ t)$, we write
$$
\|f\|_{L^{\Phi,w}(B)} = \|f\|_{L \log L^w(B)}, \quad \|f\|_{L^{\tilde{\Phi},w}(B)} = 
\|f\|_{\exp L^w(B)}.
$$
Note that by the fact that $t\leq \Phi(t)$ for $t>0$ and the above definition, we have
\begin{equation}\label{RelL-Log}
\|f\|_{L^{1,w}(B)} = \int_B |f(x)|w(x)dx = w(B)\cdot \frac1{w(B)} \int_B |f(x)|w(x)dx \leq \
\|f\|_{L\log L^w(B)}
\end{equation}
holds for any ball $B.$

We now present the definition of the generalized weighted Morrey spaces and the generalized 
weighted weak Morrey spaces which will become the spaces of our interest in this article. 
We also define the generalized weighted morrey spaces Log-type. 
The spaces will be used for boundedness of $S_\alpha$ and $S=S_0$ for the case $p=1.$ 
Moreover, we also introduce the generalized weighted mixed-Morrey spaces.

\begin{defn}\label{wgms}
Let $1 \leq p <\infty, w \in A_p,$ and $\psi$ be a positive function on $\mathbb{R}^n 
\times \mathbb{R}^+$.
The {\it generalized weighted Morrey space} $\mathcal{M}^{p,w}_\psi$ is the set of
all functions $f \in L_{\rm loc}^{p,w}$ such that
$$
\|f\|_{\mathcal{M}^{p,w}_\psi}
= \sup_{a \in \mathbb{R}^n, r>0} \frac{1}{\psi (a,r)} \frac{1}{w(B(a,r))^{\frac{1}{p}}} 
\|f\|_{L^{p,w}(B(a,r))}<\infty
$$
\end{defn}

\begin{defn}\label{wgwms}
Let $1 \leq p < \infty, w \in A_p,$ and $\psi$ be a positive function defined on 
$\mathbb{R}^n \times  \mathbb{R}^+.$
The {\it generalized weighted weak Morrey space} $W \mathcal{M}^{p, w}_\psi$
is the set of all functions $f\in L_{\rm loc}^{p,w}$ such that
$$
\|f\|_{W \mathcal{M}^{p, w}_\psi}
= \sup_{a \in \mathbb{R}^n, r>0} \frac{1}{\psi (a,r)} \frac{1}{w(B(a,r))^{\frac{1}{p}}} 
\|f\|_{WL^{p,w}(B(a,r))}<\infty.
$$
\end{defn}

\begin{remark}
There are some other variations of generalized weighted Morrey spaces and generalized 
weighted weak Morrey spaces and their specific conditions, as well as their relations with 
some classical operators such as in 
\cite{Rosenthal, Rosenthal2, Duoandikoetxea, Nakamura, NakamuraSawanoTanaka2, NakamuraSawanoTanaka}.
\end{remark}

\begin{defn}\label{wgwms-log}
Let $w\in A_1$ and $\psi$ be a positive function defined on $\mathbb{R}^n \times  
\mathbb{R}^+.$
The {\it generalized weighted weak Morrey space Log-type} $\mathcal{M}^{L \log L, w}_\psi$
is the set of all functions $f\in L_{\rm loc}^{\Phi,w}$ such that
$$
\|f\|_{\mathcal{M}^{\Phi, w}_\psi} = \sup_{a \in \mathbb{R}^n, r>0} \frac{1}{\psi (a,r)} 
\|f\|_{L^{\Phi,w}(B(a,r))} = 
\sup_{a \in \mathbb{R}^n, r>0} \frac{1}{\psi (a,r)} \|f\|_{L \log L^w(B(a,r))}<\infty.
$$
\end{defn}

\begin{defn}\label{Mixed-gwms}
Let $1\leq p, q <\infty,$ $w\in A_p,$ and $\psi$ be a positive function defined on 
$\mathbb{R}^n \times \mathbb{R}^+.$ Suppose that $T>0$. The {\it generalized weighted 
mixed-Morrey space} $\mathcal{M}^{q, w}_\psi(0,T,\mathcal{M}_\varphi^{p,w})$
is the set of all functions $f$ on $\mathbb{R}^n \times (0,T)$ such that
$$
\|f\|_{\mathcal{M}^{q, w}_\psi (0,T,\mathcal{M}_\varphi^{p,w})} = \sup_{a \in \mathbb{R}^n, 
r>0, t\in (0,T)} \frac{1}{\psi(a,r)} 
\frac{1}{w(B(a,r))^\frac1q} \left(\int_{(0,T)\cap (t_0 - r, t_0 + r)} \|f(\cdot, t)
\|^q_{\mathcal{M}^{p,w}_{\varphi}} dt\right)^\frac1q < \infty.
$$
\end{defn}

\begin{defn}\label{Mixed-wgwms}
Let $1\leq p, q <\infty,$ $w\in A_p,$ and $\psi$ be a positive function defined on 
$\mathbb{R}^n \times \mathbb{R}^+.$ Suppose that $T>0$. The {\it generalized weighted 
mixed-weak Morrey space} $W\mathcal{M}^{q, w}_\psi(0,T,\mathcal{M}_\varphi^{p,w})$
is the set of all functions $f$ on $\mathbb{R}^n \times (0,T)$ such that
$$
\|f\|_{W\mathcal{M}^{q, w}_\psi (0,T,\mathcal{M}_\varphi^{p,w})} = \sup_{a \in 
\mathbb{R}^n, r>0, t\in (0,T)} \frac{1}{\psi(a,r)} 
\frac{1}{w(B(a,r))^\frac1q} \left(\int_{(0,T)\cap (t_0 - r, t_0 + r)} \|f(\cdot, t)
\|^q_{W\mathcal{M}^{p,w}_{\varphi}} dt\right)^\frac1q<\infty.
$$
\end{defn}

\begin{defn}\label{Mixed-wgwms-Log}
Let $1\leq p, q <\infty,$ $w\in A_p,$ and the positive functions $\psi, \varphi$ which are 
defined on $\mathbb{R}^n \times \mathbb{R}^+.$ Suppose that $T>0$.
The {\it generalized weighted mixed-weak Morrey space Log type} $\mathcal{M}^{q, 
w}_\varphi (0,T,\mathcal{M}^{L \log L, w}_\psi)$
is the set of all functions $f$ on $\mathbb{R}^n \times (0,T)$ such that
$$
\|f\|_{\mathcal{M}^{q, w}_\varphi (0,T,\mathcal{M}_\psi^{L \log L,w})} = \sup_{a \in 
\mathbb{R}^n, r>0, t\in (0,T)} \frac{1}{\varphi(a,r)} \frac{1}{w(B(a,r))^\frac1q} 
\left(\int_{(0,T)\cap (t_0 - r, t_0 + r)} \|f(\cdot, t)\|^q_{\mathcal{M}^{L \log L, w}_\psi} dt\right)^\frac1q
$$
is finite.
\end{defn}

\begin{remark}\label{remark-mixed-morrey}
The definition \ref{Mixed-gwms} is generalization of Mixed-Morrey spaces as in 
\cite{Ragusa2017}. It is easy to see that if $N$ is any operator such that
$$
\|Nf(\cdot,t)\|_{\mathcal{M}_{\psi_2}^{p,w_2}} \leq C 
\|f(\cdot,t)\|_{\mathcal{M}_{\psi_1}^{q,w_1}}; \quad f \in \mathcal{M}_{\psi_1}^{q,w_1}, 
t\in (0,T),
$$
then $N$ is also bounded from $ \mathcal{M}^{q, w}_\psi(0,T,\mathcal{M}_{\psi_1}^{q,w_1}
(\Omega))$ to $\mathcal{M}^{p, w_2}_\psi(0,T,\mathcal{M}_{\psi_2}^{p,w_2}(\Omega))$.
\end{remark}

For a measurable subset $\Omega$ of $\mathbb{R}^n$, write $\|f\|_{\mathcal{M}_\psi^{p,w}
(\Omega)}=\|f\cdot\mathcal{X}_\Omega\|_{\mathcal{M}_\psi^{p,w}}$ and we write similar for 
the other type Morrey spaces. 
With the definitions, we shall investigate the boundedness of sublinear operators generated 
by Ries potentials as well as those 
generated by Calder\`on-Zygmund operators on these spaces in the next section.

Next, we turn to the BMO functions. Recall that space $BMO = BMO(\mathbb{R}^n)$ consists of 
all functions $b$ on $\mathbb{R}^n$ such that
$$
\|b\|_* := \sup_{(a,r)\in \mathbb{R}^n \times \mathbb{R}^+} \frac{1}{|B(a,r)|} \int_{B(a,r)} 
|b(y) - b_{B(a,r)}| dy <\infty
$$
where
$
b_{B(a,r)} = \frac{1}{|B(a,r)|} \int_{B(a,r)} b(y) dy.
$

\begin{lemma}\label{BMO1} \cite{Janson}
Let $b\in BMO$. Then, there exists a positive constant $C$ such that
$$
|b_{B(a,r)} - b_{B(a,t)}| \leq C \|b\|_* \ln \frac{t}{r}, \quad 0<2r<t, a \in \mathbb{R}^n.
$$
\end{lemma}

\begin{lemma}\label{BMO2} \cite{Hu}
Let $0<p<\infty$, $w\in A_\infty$, and $b\in BMO$. Then, there exists a positive constant $C>0$ such that
$$
\left(\int_{B(a,t)} |b(y)-b_{B(a,r)}|^p w(y) dy\right)^\frac1p \leq C \|b\|_* 
w(B(a,t))^\frac1p \left(1+\left|\ln 
\frac{t}{r}\right|\right), \quad a \in\mathbb{R}^n, t, r >0
$$
\end{lemma}

\begin{cor}\label{BMO3} \cite{Wang2013}
Let $0<p<\infty$, $w\in A_\infty$, and $b\in BMO$. Then, there exists a positive constant 
$C>0$ such that for any ball $B$ on $\mathbb{R}^n$,
$$
\left(\int_B |b(y)-b_B|^p w(y) dy\right)^\frac1p \leq C \|b\|_* w(B)^\frac1p
$$
\end{cor}

\begin{lemma}\label{Zhang}\cite{Zhang}
Let $w\in A_1$ and $b\in BMO$. Then, for any ball $B$ on $\mathbb{R}^n.$
$$
\|b-b_{B}\|_{\exp L^w(B)} \leq C \|b\|_*.
$$
\end{lemma}

We say that the commutator $[b,B]$, generated by an operator $B$ and a BMO 
function $b$, is bounded from $L^{p,w_1}$ to $L^{q,w_2}$ whenever there exists a constant 
$C>0$ such that
$$
\|[b,B]f\|_{L^{q,w_2}} \leq C \|b\|_* \|f\|_{L^{p,w_1}},\quad f \in L^{p,w_1}.
$$
We end this section by providing some lemmas which will be used later.

\begin{lemma} \label{lemma-boundedness}
Suppose that $u_1$, $u_2$, $v_1$, and $v_2$  are positive functions on $\mathbb{R}^n \times 
\mathbb{R}^+$ and $h:\mathbb{R}^+\times \mathbb{R}^+ \to \mathbb{R}^+$. Suppose that 
$g:\mathbb{R}^n \times \mathbb{R}^+ \to [0,\infty)$ is increasing. If 
$$
\sup_{(a,r) \in \mathbb{R}^n \times \mathbb{R}^+} \frac{1}{\psi_2(a,r)}\int_r^{\infty} 
h(t,r) \frac{\inf_{t\leq s<\infty} u_1(a,l) v_1(a,l)}{u_2(a,t)} \frac{dt}{t}\le C<\infty, 
$$
then
$$
\sup_{(a,r) \in \mathbb{R}^n \times \mathbb{R}^+} \int_r^{\infty} h(t,r) u_2(a,t)^{-1} 
g(a,t) \frac{dt}{t} \leq C v_2 (a,r) \sup_{(a,t)\in \mathbb{R}^n \times \mathbb{R}^+} 
\frac{g(a,t)}{u_1(a,t)v_1(a,t)}.
$$
\end{lemma}

\begin{proof}
By the assumption, we have that
$$
\begin{aligned}
	\int_r^{\infty} h(t,r) u_2(a,t)^{-1} g(a,t) \frac{dt}{t}
		& = \int_r^{\infty} h(t,r) \frac{\inf_{t \leq l<\infty} u_1(a,l) v_1(a,l)}{\inf_{t \leq l<\infty} u_1(a,l) v_1(a,l)} u_2(a,t)^{-1} g(a,t) \frac{dt}{t} \\
		& = \int_r^\infty h(t,r) \frac{\inf_{t \leq l<\infty} u_1(a,l) v_1(a,l)}{u_2(a,t)} \sup_{t\leq l<\infty} \frac{1}{u_1(a,l)} \frac{1}{v_1(a,l)} g(a,l) \frac{dt}{t} \\
		& = \int_r^\infty h(t,r) \frac{\inf_{t\leq l<\infty} u_1(a,l) v_1(a,l)}{u_2(a,t)} \frac{dt}{t} \cdot \sup_{(a,t)\in \mathbb{R}^n \times \mathbb{R}^+} \frac{g(a,t)}{u_1(a,t) v_1(a,t)} \\
		& \leq C v_2 (a,r) \sup_{(a,t)\in \mathbb{R}^n \times \mathbb{R}^+} \frac{g(a,t)}{u_1(a,t)v_1(a,t)}
\end{aligned}
$$
for $(a,r)\in \mathbb{R}^n \times \mathbb{R}^+$. This proves the Lemma.
\end{proof}

\begin{lemma}\label{f-1}
Let $w$ be a weight on $\mathbb{R}^n$, $0 < \alpha < n$, $1\leq p < n/\alpha$, and 
$1/q=1/p-\alpha/n$. Let $1\le s<p$ for $1<p<n/\alpha$ and $s=1$ for $p=1$. Then,
$$
r^{-\frac{\alpha}{n}+\frac1s} \leq C w_2(B(a,r))^\frac1q \|w_2^{-1}\|_{L^{ps/(p-s)}(B(a,r))}, \quad (a,r) \in \mathbb{R}^n \times \mathbb{R}^+,
$$
for $1<p<n/\alpha$, and
$$
r^{-\frac{\alpha}{n}+1} \leq C w_2(B(a,r))^\frac1q \|w_2^{-1}\|_{L^\infty(B(a,r))}, \quad (a,r) \in \mathbb{R}^n \times \mathbb{R}^+.
$$
\end{lemma}

\begin{proof}
First suppose that $1<p<\infty$. Since $p<q$ and $p'<s/(p-s)$, by H\"older 
inequality we obtain that
$$
\begin{aligned}
      1 & = \frac{1}{|B(a,r)|} \int_{B(a,r)} dx= \frac{1}{|B(a,r)|} \int_{B(a,r)} w_2(y) 
 \cdot w_2(y)^{-1} dy \\
	& \leq \frac{1}{B|(a,r)|} \left(\int_{B(a,r)} w_2^p(y) dy\right)^\frac1p 
 \left(\int_{B(a,r)} w_2^{-p'}(y) dy\right)^{\frac{1}{p'}}\\
	& \leq \left(\frac{1}{B|(a,r)|} \int_{B(a,r)} w_2^p(y) dy\right)^\frac1p 
 \left(\frac{1}{B|(a,r)|} \int_{B(a,r)} w_2^{-p'}(y) dy\right)^{\frac{1}{p'}}\\
	& \leq \left(\frac{1}{B|(a,r)|} \int_{B(a,r)} w_2^q(y) dy\right)^\frac1q 
 \left(\frac{1}{B|(a,r)|} \int_{B(a,r)} w_2^{-\frac{ps}{p-s}}(y) dy\right)^{\frac{p-s}{ps}},
\end{aligned}
$$
which implies that
$$
|B(a,r)|^{\frac1q+\frac1s-\frac1p} \leq w_2(B(a,r))^\frac1q 
\|w_2^{-1}\|_{L^{ps/(p-s)}(B(a,r))}
$$
and
$$
r^{-\frac{\alpha}{n}+\frac1s} \leq C w_2(B(a,r))^\frac1q 
\|w_2^{-1}\|_{L^{ps/(p-s)}(B(a,r))}, \quad (a,r)\in \mathbb{R}^n \times \mathbb{R}^+.
$$
	
For $p=1$, we have
$$
\begin{aligned}
 1 & = \frac{1}{|B(a,r)|} \int_{B(a,r)} \leq \frac{1}{B|(a,r)|} \left(\int_{B(a,r)} w_2(y) 
dy\right) \|w_2^{-1}\|_{L^\infty(B(a,r))}\\ 
   & \leq \frac{1}{B|(a,r)|} \left(\int_{B(a,r)} w_2(y)^q dy\right)^\frac1q
  \|w_2^{-1}\|_{L^\infty(B(a,r))},
\end{aligned}
$$
and so
$$
r^{-\frac\alpha{n} + 1} \leq C w_2^q(B(a,r))^\frac1q \|w_2^{-1}\|_{L^\infty (B(a,r))}, 
\quad (a,r) \in \mathbb{R}^n \times \mathbb{R}^+,
$$
which completes the proof of the lemma.
\end{proof}

\section{Norm Estimates for Sublinear Operators Generated by Riesz Potentials and Their 
Commutators on Generalized Weighted Morrey Spaces}

In this section, we prove the boundedness of the sublinear operators generated by Riesz 
potentials on generalized weighted Morrey spaces with different weights, namely Theorems 
\ref{Main-Riesz} and \ref{Main-Riesz-Mixed}. Moreover, we also prove the boundedness of 
their commutators of the sublinear operators with a BMO function on generalized weighted 
Morrey spaces with different weights, namely Theorems \ref{Main-Comm-Riesz} and 
\ref{Main-Comm-Riesz-Mixed}. Before we do so, we prove the following estimates.

\begin{prop}\label{est-sub-Riesz}
Let $0 < \alpha < n, 1 \le p < n/\alpha, 1/q = 1/p - \alpha/n$, $(w_1^s, w_2^s)\in
A_{p/s,q/s}$, and $w_2^s \in A_{p/s,q/s}$, where $1\le s<p$ for $1<p<n/\alpha$ and $s=1$ 
for $p=1$. 
If $S_\alpha$ is bounded from $L^{p,w_1^p}$ to $L^{q,w_2^q}$ for $1<p<n/\alpha$ and from 
$L^{1,w_1}$ to $WL^{q,w_2^q}$, then there exists a constant $C>0$ such that for every 
$(a,r) \in \mathbb{R}^n \times \mathbb{R}^+$ and $f \in L_{\rm loc}^{p,w_1^p}$,
$$
\|S_\alpha f\|_{L^{q,w_2^q}(B(a,r))} \leq C_1 w_2^q(B(a,r))^{\frac{1}{q}} \int_r^\infty 
w_2^q(B(a,t))^{-\frac{1}{q}} \|f\|_{L^{p,w_1^p}B(a,t)} \frac{dt}{t},
$$
for $1<p<n/\alpha,$ and
$$
\|S_\alpha f\|_{WL^{q,w_2^q}(B(a,r))} \leq C_2 w_2^q(B(a,r))^{\frac{1}{q}} \int_r^\infty 
w_2^q(B(a,t))^{-\frac{1}{q}} \|f\|_{L^{1,w_1}B(a,t)} \frac{dt}{t}.
$$
\end{prop}

\begin{proof}
Given $a\in \mathbb{R}^n$ and $r>0$, we decompose $f$ as $f := f_1 + f_2$ where 
$f_1 :=f\cdot \mathcal{X}_{B(a,2r)}$. Note that $f_1 \in L^{p,w_1^p}.$ Now, since 
$S_\alpha$ is bounded from $L^{p,w_1^p}$ to $L^{q,w_2^q}$ and from $L^{1,w}$ to $WL^{q,w^q}$,
$$
\|S_\alpha f_1\|_{L^{q,w_2^q}(B(a,r))} \leq \|S_\alpha f_1\|_{L^{q,w_2^q}} \leq C 
\|f_1\|_{L^{p,w_1^p}} = C \|f\|_{L^{p,w_1^p}(B(a,2r))}
$$
for $1<p<n/\alpha$, and
$$
\|S_\alpha f_1\|_{WL^{q,w_2^q}(B(a,r))} \leq \|S_\alpha f_1\|_{L^{q,w_2^q}} \leq C 
\|f_1\|_{L^{1,w_1}} = C \|f\|_{L^{1,w_1}(B(a,2r))}.
$$
Thus, by Lemma \ref{f-1} and the assumption that $w_2^s \in A_{p/s,q/s}$ for 
$1< p< n/\alpha$, we have
$$
\begin{aligned}
\|f\|_{L^{p,w^p}(B(a,2r))} 
& = C r^{\frac{n}s - \alpha} \|f\|_{L^{p,w^p}(B(a,2r))} \int_{2r}^\infty \frac{t}{t^{n/s-
\alpha+1}}\\
& \leq Cr^{\frac{n}{s}-\alpha} \int_{2r}^\infty \|f\|_{L^{p,w^p}(B(a,t))} \frac{dt}{t^{n/s-
\alpha+1}}\\
& \leq w_2(B(a,r))^\frac1q \|w_2^{-1}\|_{L^{ps/(p-s)}(B(a,r))} \int_{2r}^\infty 
\|f\|_{L^{p,w^p}(B(a,t))} \frac{dt}{t^{n/s-\alpha+1}}\\
& \leq w_2(B(a,r))^\frac1q \int_{2r}^\infty \|f\|_{L^{p,w^p}(B(a,t))} 
\|w_2^{-1}\|_{L^{ps/(p-s)}(B(a,t))} \frac{dt}{t^{n/s-\alpha+1}} \\
& \leq w_2(B(a,r))^\frac1q \int_{2r}^\infty \|f\|_{L^{p,w^p}(B(a,t))} w_2^q(B(a,t))^{
-\frac1q} \frac{dt}{t}.\\
\end{aligned}
$$
Hence,
$$
\|S_\alpha f_1\|_{L^{q,w_2^q}(B(a,r))} \leq C w_2^q(B(a,r))^\frac1q \int_{2r}^\infty 
\|f\|_{L^{p,w_1^p}(B(a,t))} w_2^q(B(a,t))^{-\frac1q} \frac{dt}{t}
$$
for $1<p<n/\alpha$, and
$$
\|S_\alpha f_1\|_{L^{q,w_2^q}(B(a,r))} \leq w_2^q(B(a,r))^\frac1q \int_{2r}^\infty 
\|f\|_{L^{1,w_1}(B(a,t))} w_2^q(B(a,t))^{-\frac1q} \frac{dt}{t}.
$$
	
For $f_2$, it follows from the assumption that
$$
|S_\alpha f_2 (x)| \leq C \sum_{k=1}^\infty \left(\frac{1}{|B(a,2^{k+1}r)|^{1-\frac{\alpha 
s}{n}}} \int_{B(a,2^{k+1}r)}|f(y)|^s dy\right)^\frac1s.
$$
Now suppose that $1<p<n/\alpha$. Then, H\"older inequality implies that
$$
\begin{aligned}
		& \left(\int_{B(a,2^{k+1}r)} |f(y)|^s dy\right)^\frac1s = 
\left(\int_{B(a,2^{k+1}r)} |f(y)|^s w_1 (y)^s w_1(y)^{-s} dy\right)^\frac1s\\
		& \leq \left(\int_{B(a,2^{k+1}r)} \left(|f(y)|^s\right)^{\frac{p}{s}} 
\left(w_1(y)^s\right)^\frac{p}{s}\right)^{\left(\frac{s}{p}\right)
\left(\frac1s\right)} \left(\int_{B(a,2^{k+1}r)}  \left(w_1(y)^{-s}\right)^\frac{p}{p-s}\right)^{\left(\frac{p-s}{p}\right)\left(\frac1s\right)}\\
		& = \left(\int_{B(a,2^{k+1}r)} |f(y)|^p w_1(y)^p\right)^\frac1p 
\left(\int_{B(a,2^{k+1}r)}  w_1(y)^{-\frac{ps}{p-s}}\right)^{\left(
\frac{p-s}{ps}\right)}\\
		& = \|f\|_{L^{p,w_1^p}(B(a,2^{k+1}r))} 
\|w_1^{-1}\|_{L^{ps/(p-s)}(B(a,2^{k+1}r))}.
\end{aligned}
$$
Hence, for $x \in B(a,r)$ by the assumption $(w_1^s, w_2^s) \in A_{p/s,q/s}$ we obtain that
$$
\begin{aligned}
|S_\alpha f_2 (x)| & \leq C \sum_{k=1}^\infty \frac{1}{|B(a,2^{k+1}r)|^{\frac1s-
\frac{\alpha }{n}}} \|f\|_{L^{p,w_1^p}(B(a,2^{k+1}r))} 
\|w_1^{-1}\|_{L^{ps/(p-s)}(B(a,2^{k+1}r))} \\
		& \leq C \sum_{k=1}^\infty \left(2^{k+1}r\right)^{\alpha-\frac{n}{s}} 
\|f\|_{L^{p,w_1^p}(B(a,2^{k+1}r))} \|w_1^{-1}\|_{L^{ps/(p-s)}(B(a,2^{k+1}r))} \\
		& = C \sum_{k=1}^\infty \int_{2^{k+1}r}^{2^{k+2}r} \frac{dt}{t^{-\alpha+
\frac{n}{s}+1}} \cdot \|f\|_{L^{p,w_1^p}(B(a,2^{k+1}r))} 
\|w_1^{-1}\|_{L^{ps/(p-s)}(B(a,2^{k+1}r))} \\
		& = C \sum_{k=1}^\infty \int_{2^{k+1}r}^{2^{k+2}r} \|f\|_{L^{p,w_1^p}(B(a,t))} 
\|w_1^{-1}\|_{L^{ps/(p-s)}(B(a,t))} \frac{dt}{t^{-\alpha+\frac{n}{s}+1}}\\
		& \leq C \int_{2r}^\infty \|f\|_{L^{p,w_1^p}(B(a,t))} w_2^q(B(a,t))^{-\frac1q} 
\frac{dt}{t}.
\end{aligned}
$$
Thus,
$$
\|S_\alpha f_2\|_{L^{q,w_2^q}(B(a,r))} \leq C w_2^q (B(a,r))^\frac1q \int_{2r}^\infty 
\|f\|_{L^{p,w_1^p}(B(a,t))} w_2^q(B(a,t))^{-\frac1q} \frac{dt}{t}.
$$
for $1<p<n/\alpha.$ 

Next, for $p=1,$ we have
$$
\begin{aligned}
\left(\int_{B(a,2^{k+1}r)} |f(y)|^s dy\right)^\frac1s & = \int_{B(a,2^{k+1}r)} |f(y)| 
w_1(y) w_1(y)^{-1} dy\\
	& \leq \|f\|_{L^{1,w_1}(B(a,2^{k+1}r))} \|w_1^{-1}\|_{L^\infty(B(a,2^{k+1}r))}.
\end{aligned}
$$
For $x \in B(a,r)$ we obtain that
$$
\begin{aligned}
|S_\alpha f_2 (x)| & \leq C \sum_{k=1}^\infty \frac{1}{|B(a,2^{k+1}r)|^{1-
\frac{\alpha}{n}}} \|f\|_{L^{1,w_1}(B(a,2^{k+1}r))} 
\|w_1^{-1}\|_{L^\infty(B(a,2^{k+1}r))} \\
		& \leq C \sum_{k=1}^\infty \left(2^{k+1}r\right)^{\alpha-n} 
\|f\|_{L^{1,w_1}(B(a,2^{k+1}r))} \|w_1^{-1}\|_{L^\infty (B(a,2^{k+1}r))} \\
		& = C \sum_{k=1}^\infty \int_{2^{k+1}r}^{2^{k+2}r} \frac{dt}{t^{-\alpha + n + 
1}} \cdot \|f\|_{L^{1,w_1}(B(a,2^{k+1}r))} \|w_1^{-1}\|_{L^\infty(B(a,2^{k+1}r))} \\
		& = C \sum_{k=1}^\infty \int_{2^{k+1}r}^{2^{k+2}r} \|f\|_{L^{1,w_1}(B(a,t))} 
\|w_1^{-1}\|_{L^\infty(B(a,t))} \frac{dt}{t^{-\alpha+ n + 1}}\\
		& \leq C \int_{2r}^\infty \|f\|_{L^{1,w_1}(B(a,t))} w_2^q(B(a,t))^{-\frac1q} 
\frac{dt}{t},
\end{aligned}
$$
and
$$
\|S_\alpha f_2\|_{WL^{q,w_2^q}(B(a,r))} \leq \|S_\alpha f_2\|_{L^{q,w_2^q}(B(a,r))}\leq C 
w_2^q (B(a,r))^\frac1q \int_{2r}^\infty \|f\|_{L^{1,w_1}(B(a,t))} w_2^q(B(a,t))^{-\frac1q} 
\frac{dt}{t}.
$$
Therefore,
$$
\|S_\alpha f\|_{L^{q,w_2^q}(B(a,r))} \leq C_1 w_2^q(B(a,r))^{\frac{1}{q}} \int_r^\infty 
w_2^q(B(a,t))^{-\frac{1}{q}} 
\|f\|_{L^{p,w_1^p}B(a,t)} \frac{dt}{t},
$$
for $1<p<n/\alpha,$ and
$$
\|S_\alpha f\|_{WL^{q,w_2^q}(B(a,r))} \leq C_2 w_2^q(B(a,r))^{\frac{1}{q}} \int_r^\infty 
w_2^q(B(a,t))^{-\frac{1}{q}} \|f\|_{L^{1,w_1}B(a,t)} \frac{dt}{t},
$$
as desired.
\end{proof}

\begin{prop}\label{est-comm-sub-riesz}
Let $0 < \alpha < n, 1 < p < n/\alpha, 1/q = 1/p - \alpha/n$, $(w_1^s, w_2^s) \in 
A_{p/s,q/s}$ and $w_2^s \in A_{p/s, q/s}$ where $1\le s<p$, and $w_1 \in A_\infty$. 
If $[b, S_\alpha]$ is bounded from $L^{p,w_1^p}$ to $L^{q,w_2^q}$, then there exists a 
constant $C>0$ such that for every $(a,r) \in \mathbb{R}^n \times \mathbb{R}^+$ and 
$f \in L_{\rm loc}^{p,w_1^p}$,
$$
\|[b, S_\alpha ]f\|_{L^{q,w_2^q}(B(a,r))} \leq C \|b\|_* w_2^q(B(a,r))^{\frac{1}{q}} 
\int_r^\infty w_2^q(B(a,t))^{-\frac{1}{q}} \|f\|_{L^{p,w_1^p}B(a,t)} \frac{dt}{t}.
$$
\end{prop}

\begin{proof}
Let $a\in \mathbb{R}^n$ and $r>0$, and we represent the function $f \in L_{\rm 
loc}^{p,w_1^p}$ as $f := f_1 + f_2$ where $f_1 :=
f\cdot \mathcal{X}_{B(a,2r)}$. Since  $f_1 \in L^{p,w_1^p}$ and $[b, S_\alpha]$ is bounded 
from $L^{p,w_1^p}$ to $L^{q,w_2^q}$, we find that
$$
\|[b, S_\alpha] f_1\|_{L^{q,w_2^q}(B(a,r))} \leq \|[b, S_\alpha ]f f_1\|_{L^{q,w_2^q}} \leq 
C \|b\|_*\|f_1\|_{L^{p,w_1^p}} = C \|b\|_*\|f\|_{L^{p,w_1^p}(B(a,2r))},
$$
and by Lemma \ref{f-1},
\begin{equation}\label{Comf1}
\|[b, S_\alpha] f_1\|_{L^{q,w_2^q}(B(a,r))} \leq C \|b\|_* w_2^q(B(a,r))^\frac1q 
\int_{2r}^\infty \|f\|_{L^{p,w_1^p}(B(a,t))} w_2^q(B(a,t))^{-\frac1q} \frac{dt}{t}.
\end{equation}
	
For $f_2$, by (\ref{commutator})
$$
|[b, S_\alpha] f_2 (x)| \leq C |b(x) - b_{B(a,r)}| |S_\alpha (f_2) (x)| + |S_\alpha 
((b_{B(a,r)}-b)(f))(x)|, \quad x \in B(a,r).
$$
	
We write
$$
J_1 (x) := |b(x) - b_{B(a,r)}| |S_\alpha (f_2) (x)|, \quad J_2 (x) := |S_\alpha 
((b_{B(a,r)}-b)(f))(x)|.
$$
We shall find the estimates for $J_1$ and $J_2$. As in Theorem \ref{est-sub-Riesz},
$$
J_1 (x) \leq C |b(x) - b_{B(a,r)}| \int_{2r}^{\infty} \|f\|_{L^{p,w_1^p}(B(a,t))} 
w_2^q(B(a,t))^{-\frac1q} \frac{dt}{t}.
$$
Since $w_2^s \in A_{p/s,q/s}$, we have that $w_2^q \in A_q$ and by Lemma \ref{BMO3}, it 
follows that
\begin{equation}\label{J1}
\|J_1\|_{L^{q,w^q}(B(a,r))} \leq C \|b\|_* w_2^q(B(a,r))^\frac1q \int_{2r}^\infty 
\|f\|_{L^{p,w_1^q}(B(a,t))} w_2^q(B(a,t))^{-\frac1q} \frac{dt}{t}
\end{equation}
	
By (\ref{sublinear}), we obtain that
$$
J_2 (x) \leq C \sum_{k=1}^\infty (2^{k+1}r)^{-\frac{n}s + \alpha} 
\left(\int_{B(a,2^{k+1}r)} |(b(y) - b_{B(a,r)})f(y)|^s\right)^\frac1s, \quad x \in B(a,r).
$$
H\"older inequality then implies that
$$
\left(\int_{B(a,2^{k+1}r)} |(b(y) - b_{B(a,r)})f(y)|^s\right)^\frac1s \leq 
\|f\|_{L^{p,w_1^p}(B(a,2^{k+1}r))} \left\|(b-b_{B(a,r)})w^{-1} \right\|_{L^{ps/(p-s)}
(B(a,2^{k+1}r))}.
$$
	
Thus, by the assumption that $w_1 \in A_\infty$ and Lemma \ref{f-1},
$$
\begin{aligned}
J_2 (x) & \leq C \sum_{k=1}^\infty (2^{k+1}r)^{-\frac{n}s + \alpha} 
\|f\|_{L^{p,w_1^p}(B(a,2^{k+1}r))} \left\| (b-b_{B(a,r)})w^{-1} 
\right\|_{L^{ps/(p-s)}(B(a,2^{k+1}r))} \\
	& = C \int_{2^{k+1}r}^{2^{k+2}r} \frac{dt}{t^{-\alpha+n+1}} 
 \|f\|_{L^{p,w_1^p}(B(a,2^{k+1}r))} \|(b-b_{B(a,r)}) 
 w_1^{-1}\|_{L^{\frac{ps}{p-s}}(B(a,2^{k+1}r))} \\
	& \leq C \int_{2r}^\infty \|f\|_{L^{p,w_1^p}(B(a,t))} \|
 (b-b_{B(a,r)})w_1^{-1}\|_{L^{\frac{ps}{p-s}}(B(a,t))} \frac{dt}{t^{-\alpha+n+1}}\\
	& \leq C \int_{2r}^\infty \|f\|_{L^{p,w_1^p} (B(a,t))} \left(\int_{B(a,t)} 
 |b(y) -b_{B(a,r)}|^{\frac{ps}{p-s}} w_1(y)^{-\frac{ps}{p-s}} dy\right)^{\frac{p-s}{ps}} 
 \frac{dt}{t^{-\alpha+n+1}}\\
	& = C \int_{2r}^\infty \|f\|_{L^{p,w_1^p}(B(a,t))} \|b\|_* \left(1 + \ln 
 \frac{t}r \right) \left(w_1^{-\frac{ps}{p-s}}(B(a,t))\right)^{\frac{p-s}{ps}} 
 \frac{dt}{t^{-\alpha+n+1}} \\
	& = C \|b\|_* \int_{2r}^{\infty} \left(1+ \ln \frac{t}r \right) 
 \|f\|_{L^{p,w_1^p}(B(a,t))} \|w_1^{-1}\|_{L^{\frac{ps}{p-s}}(B(a,t))} 
 \frac{dt}{t^{-\alpha+n+1}} \\
	& \leq C \|b\|_* \int_{2r}^\infty \left(1+\ln \frac{t}r\right) 
 \|f\|_{L^{p,w_1^p}(B(a,t))} w_2^q(B(a,t))^{-\frac1q} \frac{dt}{t}.
\end{aligned}
$$
It follows that
\begin{equation} \label{J-2}
\|J_2\|_{L^{q,w_2^q}(B(a,r))}\leq C w_2^q(B(a,r)) \int_{2r}^\infty \left(1+\ln 
\frac{t}r\right) \|f\|_{L^{p,w_1^p}(B(a,t))} w_2^q(B(a,t))^{-\frac1q} \frac{dt}{t}.
\end{equation}
By combining (\ref{J1}) and (\ref{J-2}), we then obtain
\begin{equation} \label{Comf2}
\|[b,S_\alpha]f_2\|_{L^{q,w_2^q}(B(a,r))} \leq C \|b\|_* \int_{2r}^\infty \left(1+\ln 
\frac{t}r\right) \|f\|_{L^{p,w_1^p}(B(a,t))} w_2^q(B(a,t))^{-\frac1q} \frac{dt}{t}.
\end{equation}
Now	(\ref{Comf1}) and (\ref{Comf2}) imply that
$$
\|[b,S_\alpha]f\|_{L^{q,w_2^q}(B(a,r))} \leq C \|b\|_* \int_{2r}^\infty \left(1+\ln 
\frac{t}{r}\right) \|f\|_{L^{p,w_1^p}(B(a,t))} w_2^q(B(a,t))^{-\frac1q} \frac{dt}{t},
$$
which proves the proposition.
\end{proof}

Now, we are ready to prove Theorem \ref{Main-Riesz}, Theorem \ref{Main-Comm-Riesz}, Theorem \ref{Main-Riesz-Mixed}, and 
Theorem \ref{Main-Comm-Riesz-Mixed}.

\begin{proof}[\textbf{Proof of Theorem \ref{Main-Riesz}}]
	By Lemma \ref{lemma-boundedness}, we have that for $(a,r) \in \mathbb{R}^n \times \mathbb{R}^n$ the following inequality holds.
	$$
	\begin{aligned}
		\int_r^{\infty} w_2^q(B(a,t))^{-\frac{1}{q}} \|f\|_{L^{p,w_1^p}(B(a,t))} \frac{dt}{t}
		\leq \|f\|_{\mathcal{M}_{\psi_1}^{p,w_1^p}} \int_r^{\infty} \frac{\inf_{t\leq s<\infty} \psi_1(a,s)w_1^p(B(a,s))^\frac1p}{w_2^q(B(a,s))^\frac1q} \frac{dt}{t}
	\end{aligned}
	$$
	By Proposition \ref{est-sub-Riesz}, assumptions, and the last inequality, we obtain that
	$$
	\begin{aligned}
		\|S_\alpha f\|_{\mathcal{M}_{\psi_2}^{q,w^q}} & = \sup_{(a,r) \in \mathbb{R}^n \times \mathbb{R}^+} \frac{1}{\psi_2 (a,r)} \frac{1}{w_2^q(B(a,r))^\frac1q} \|S_\alpha f\|_{L^{q,w_2^q}(B(a,r))} \\
		& \leq C \sup_{(a,r) \in \mathbb{R}^n \times \mathbb{R}^+} \frac{1}{\psi_2(a,r)} \int_{r}^\infty w_2^q(B(a,t))^{-\frac1q} \|f\|_{L^{p,w_1^p}(B(a,r))} \frac{dt}{t}\\
		& \leq C \|f\|_{\mathcal{M}_{\psi_1}^{p,w_1^p}} \sup_{(a,r)\in \mathbb{R}^n \times \mathbb{R}^+} \frac{1}{\psi_2(a,r)} \int_r^\infty \frac{\inf_{t \leq l < \infty} \psi_1 (a,l) w_1^p(B(a,l))^\frac1p}{w_2^q (B(a,t))^\frac1q} \frac{dt}{t}  \leq C \|f\|_{\mathcal{M}_{\psi_1}^{p,w_1^p}}
	\end{aligned}
	$$
	for $1<p<n/\alpha$. Meanwhile,
	$$
	\begin{aligned}
		\|S_\alpha f \|_{W\mathcal{M}_{\psi_2}^{q,w_2^q}} & = \sup_{(a,r) \in \mathbb{R}^n \times \mathbb{R}^+} \frac{1}{\psi_2(a,r)} \frac{1}{w_2^q(B(a,r))^\frac1q} \|S_\alpha f\|_{WL^{q,w_2^q} (B(a,r))} \\
		& \leq C \sup_{(a,r) \in \mathbb{R}^n \times \mathbb{R}^+} \frac{1}{\psi_2(a,r)} \int_r^\infty w_2^q(B(a,t))^{-\frac1q} \|f\|_{L^{1,w}(B(a,t))} \frac{dt}{t}\\
		& \leq C \|f\|_{\mathcal{M}_{\psi_1}^{1,w}} \sup_{(a,r)\in \mathbb{R}^n \times \mathbb{R}^+} \frac{1}{\psi_2(a,r)} \int_r^\infty \frac{\inf_{t \leq l < \infty} \psi_1 (a,l) w_1(B(a,l))}{w_2^q (B(a,r))^\frac1q} \frac{dt}{t} \leq C \|f\|_{\mathcal{M}_{\psi_1}^{1,w}}.
	\end{aligned}
	$$
	This proves the theorem.
\end{proof}

\begin{proof} [\textbf{Proof of Theorem \ref{Main-Comm-Riesz}}]
	By Lemma \ref{lemma-boundedness}, we have that for $(a,r) \in \mathbb{R}^n \times \mathbb{R}^n$ the following inequality holds.
	$$
	\begin{aligned}
		\int_r^{\infty} \left(1+\ln \frac{t}r \right) w_2^q(B(a,t))^{-\frac{1}{q}} & \|f\|_{L^{p,w_1^p}(B(a,t))} \frac{dt}{t} \\
		& \leq \|f\|_{\mathcal{M}_{\psi_1}^{p,w_1^p}} \int_r^{\infty} \left(1+\ln \frac{t}r \right) \frac{\inf_{t\leq s<\infty} \psi_1(a,s)w_1^p(B(a,s))^\frac1p}{w_2^q(B(a,s))^\frac1q} \frac{dt}{t}.
	\end{aligned}
	$$
	By Proposition \ref{est-comm-sub-riesz}, assumptions, and the last inequality, we obtain that
	$$
	\begin{aligned}
		& \|[b,S_\alpha] f\|_{\mathcal{M}_{\psi_2}^{q,w_2^q}}
		\\ & = \sup_{(a,r) \in \mathbb{R}^n \times \mathbb{R}^+} \frac{1}{\psi_2 (a,r)} \frac{1}{w_2^q(B(a,r))^\frac1q} \|[b,S_\alpha] f\|_{L^{q,w_2^q}(B(a,r))} \\
		& \leq C \sup_{(a,r) \in \mathbb{R}^n \times \mathbb{R}^+} \frac{1}{\psi_2(a,r)} \int_{r}^\infty \left(1+\ln \frac{t}r \right) w_2^q(B(a,t))^{-\frac1q} \|f\|_{L^{p,w_1^p}(B(a,r))} \frac{dt}{t}\\
		& \leq C \|f\|_{\mathcal{M}_{\psi_1}^{p,w_1^p}} \sup_{(a,r)\in \mathbb{R}^n \times \mathbb{R}^+} \frac{1}{\psi_2(a,r)} \int_r^\infty \left(1+\ln \frac{t}r \right) \frac{\inf_{t \leq l < \infty} \psi_1 (a,l) w_1^p(B(a,l))^\frac1p}{w_2^q (B(a,t))^\frac1q} \frac{dt}{t}  \\
		& \leq C \|f\|_{\mathcal{M}_{\psi_1}^{p,w_1^p}}.
	\end{aligned}
	$$
	This proves the theorem.
\end{proof}

\begin{proof}[\textbf{Proof of Theorem \ref{Main-Riesz-Mixed}}]
	One may replace $f(\cdot)$ by $f(\cdot, t)$ in Proposition \ref{est-sub-Riesz} to obtain
	$$
	\|S_\alpha f(\cdot, t)\|_{\mathcal{M}^{q,w_2^q}_{\psi_1}} \leq C \|f(\cdot, t)\|_{\mathcal{M}^{p,w_1^p}_{\psi_1}}, \quad 
f\in \mathcal{M}^{p,w_1^p}_{\psi_1}
	$$
	where $C>0$ is independent of $t$. Therefore, by Definition \ref{Mixed-wgwms} and Remark \ref{remark-mixed-morrey} we have the desired results of $S_\alpha$ on generalized weighted mixed-morrey spaces.
\end{proof}

\begin{proof}[\textbf{Proof of Theorem \ref{Main-Comm-Riesz-Mixed}}]
	One may replace $f(\cdot)$ by $f(\cdot, t)$ in Proposition \ref{est-comm-sub-riesz}, and obtain
	$$
	|[b, S_\alpha]f(\cdot, t)\|_{\mathcal{M}^{q,w_2^q}_{\psi_1}} \leq C \|b\|_* \|f(\cdot, t)\|_{\mathcal{M}^{p,w_1^p}_{\psi_1}}, 
\quad f\in \mathcal{M}^{p,w_1^p}_{\psi_1}
	$$
	where $C>0$ is independent of $t$. The conclusion then follows from Definition \ref{Mixed-wgwms} and Remark \ref{remark-mixed-morrey}.
\end{proof}

\section{Norm Estimates for Sublinear Operators Generated by Calder\`on-Zygmund Operators on 
Generalized Weighted Morrey Spaces}

We prove the boundedness for sublinear operators $S$ generated by Calder\`on-Zymgund 
operators on generalized weighted Morrey spaces and generalized weighted weak Morrey spaces. 
In addition, we prove the boundedness for commutator of the operator with BMO functions 
on generalized weighted Morrey spaces.

\begin{prop}\label{est-sub-calder}
Let $1 \leq p < \infty$ and $(w_1, w_2), w_2 \in A_{p/s}$, where $1\le s<p$ for $1<p<\infty$ 
and $s=1$ for $p=1$. If $S$ is bounded from $L^{p,w_1}$ to $L^{p,w_2}$ for $1<p<\infty$ and 
from $L^{1,w_1}$ to $WL^{1,w_2}$, then there exists a constant $C>0$ such that for every 
$(a,r) \in \mathbb{R}^n \times \mathbb{R}^+$ and $f \in L_{\rm loc}^{p,w_1}$,
$$
	\|S f\|_{L^{p,w_2}(B(a,r))} \leq C_1 w_2(B(a,r))^{\frac{1}{p}} \int_r^\infty w_2(B(a,t))^{-\frac{1}{p}} \|f\|_{L^{p,w_1}B(a,t)} \frac{dt}{t},
$$
for $1<p<\infty,$ and
	$$
	\|S f\|_{WL^{1,w_2}(B(a,r))} \leq C_2 w_2(B(a,r)) \int_r^\infty w_2(B(a,t))^{-1} \|f\|_{L^{1,w_1}B(a,t)} \frac{dt}{t}.
	$$
\end{prop}

\begin{proof}
Let $a\in \mathbb{R}^n$ and $r>0$, and we write $f := f_1 + f_2$ where $f_1 := f\cdot 
\mathcal{X}_{B(a,2r)}$. Since $T$ is bounded from $L^{p,w_1}$ to $L^{p,w_2}$ and from 
$L^{1,w_1}$ to $WL^{1,w_2}$,
$$
\|S f_1\|_{L^{p,w_2}(B(a,r))} \leq \|S f_1\|_{L^{p,w_2}} \leq C \|f_1\|_{L^{p,w_1}} = C 
\|f\|_{L^{p,w_1}(B(a,2r))}
$$
for $1<p<\infty$ and
$$
\|S f_1\|_{WL^{1,w_2}(B(a,r))} \leq \|S f_1\|_{L^{1,w_2}} \leq C \|f_1\|_{L^{1,w_1}} = C 
\|f\|_{L^{1,w_1}(B(a,2r))}.
$$
Thus, by Lemma \ref{f-1} and the assumption that $w_2 \in A_{p/s}$, we have
$$
\begin{aligned}
	\|f\|_{L^{p,w_1}(B(a,2r))} & = C r^{\frac{n}s} \|f\|_{L^{p,w_1}(B(a,2r))} 
 \int_{2r}^\infty \frac{t}{t^{n/s+1}} \leq Cr^{\frac{n}{s}} \int_{2r}^\infty 
 \|f\|_{L^{p,w_1}(B(a,t))} \frac{dt}{t^{n/s+1}}\\
	& \leq w_2(B(a,r))^\frac1p \|w_2^{-1}\|_{L^{ps/(p-s)}(B(a,r))} \int_{2r}^\infty 
 \|f\|_{L^{p,w_1}(B(a,t))} \frac{dt}{t^{n/s+1}}\\
	& \leq w_2(B(a,r))^\frac1p \int_{2r}^\infty \|f\|_{L^{p,w_1}(B(a,t))} 
 \|w_2^{-1}\|_{L^{ps/(p-s)}(B(a,t))} \frac{dt}{t^{n/s+1}} \\
	& \leq w_2(B(a,r))^\frac1p \int_{2r}^\infty \|f\|_{L^{p,w_1}(B(a,t))} 
 w_2(B(a,t))^{-\frac1p} \frac{dt}{t}. \\
\end{aligned}
$$
Hence,
$$
\|S f_1\|_{L^{p,w_2}(B(a,r))} \leq C w_2(B(a,r))^\frac1p \int_{2r}^\infty 
\|f\|_{L^{p,w_1}(B(a,t))} w_2(B(a,t))^{-\frac1p} \frac{dt}{t}
$$
for $1<p<\infty$, and
$$
\|S f_1\|_{WL^{1,w_2}(B(a,r))} \leq w_2(B(a,r)) \int_{2r}^\infty \|f\|_{L^{1,w_1}(B(a,t))} 
w_2(B(a,t))^{-1} \frac{dt}{t}.
$$
	
For $f_2$, by the assumption,
$$
|S f_2 (x)| \leq C \sum_{k=1}^\infty \left(\frac{1}{|B(a,2^{k+1}r)|} 
\int_{B(a,2^{k+1}r)}|f(y)|^s dy\right)^\frac1s.
$$
First consider the case where $1<p<\infty$ and $1\le s<p$. Then, H\"older inequality implies 
that
$$
\begin{aligned}
	& \left(\int_{B(a,2^{k+1}r)} |f(y)|^s dy\right)^\frac1s = \left(\int_{B(a,2^{k+1}r)} 
 |f(y)|^s w_1 (y)^\frac{s}{p} w_1(y)^{-\frac{s}{p}} dy\right)^\frac1s\\
	& \leq \left(\int_{B(a,2^{k+1}r)} \left(|f(y)|^s\right)^{\frac{p}{s}} 
 \left(w_1(y)^\frac{s}{p}\right)^\frac{p}{s}\right)^{\left(
 \frac{s}{p}\right)\left(\frac1s\right)} \left(\int_{B(a,2^{k+1}r)}  
 \left(w_1(y)^{-\frac{s}{p}}\right)^\frac{p}{p-s}\right)^{\left(
 \frac{p-s}{p}\right)\left(\frac1s\right)}\\
	& = \left(\int_{B(a,2^{k+1}r)} |f(y)|^p w_1(y)\right)^\frac1p 
 \left(\int_{B(a,2^{k+1}r)}  w_1(y)^\frac{-s}{p-s}\right)^{\left(\frac{p-s}{ps}\right)}\\
	& = \|f\|_{L^{p,w_1}(B(a,2^{k+1}r))} \|w_1^{-{\frac1p}}
 \|_{L^{ps/(p-s)}(B(a,2^{k+1}r))}.
\end{aligned}
$$
Hence, for $x \in B(a,r)$ by the assumption $(w_1, w_2) \in A_{p/s}$ we obtain
$$
\begin{aligned}
	|S f_2 (x)| & \leq C \sum_{k=1}^\infty \frac{1}{|B(a,2^{k+1}r)|^{\frac1s}} 
 \|f\|_{L^{p,w_1}(B(a,2^{k+1}r))} \|w_1^{-{\frac1p}}\|_{L^{ps/(p-s)}(B(a,2^{k+1}r))} \\
	& \leq C \sum_{k=1}^\infty \left(2^{k+1}r\right)^{-\frac{n}{s}} 
 \|f\|_{L^{p,w_1}(B(a,2^{k+1}r))} \|w_1^{-{\frac1p}}\|_{L^{ps/(p-s)}(B(a,2^{k+1}r))} \\
	& = C \sum_{k=1}^\infty \int_{2^{k+1}r}^{2^{k+2}r} \frac{dt}{t^{+\frac{n}{s}+1}} 
 \cdot \|f\|_{L^{p,w_1}(B(a,2^{k+1}r))} \|w_1^{-{\frac1p}}\|_{L^{ps/(p-s)}(B(a,2^{k+1}r))} \\
	& = C \sum_{k=1}^\infty \int_{2^{k+1}r}^{2^{k+2}r} \|f\|_{L^{p,w_1}(B(a,t))} 
 \|w_1^{-{\frac1p}}\|_{L^{ps/(p-s)}(B(a,t))}  \frac{dt}{t^{\frac{n}{s}+1}}\\
	& \leq C \int_{2r}^\infty \|f\|_{L^{p,w_1}(B(a,t))} w_2(B(a,t))^{-\frac1p} 
 \frac{dt}{t}.
\end{aligned}
$$
Thus,
$$
	\|S f_2\|_{L^{p,w_2}(B(a,r))} \leq C w_2 (B(a,r))^\frac1p \int_{2r}^\infty 
 \|f\|_{L^{p,w_1}(B(a,t))} w_2(B(a,t))^{-\frac1p} \frac{dt}{t}.
$$
	
Next, for the case where $p=1$ and $s=1$, we have
$$
\begin{aligned}
	\left(\int_{B(a,2^{k+1}r)} |f(y)|^s dy\right)^\frac1s & = \int_{B(a,2^{k+1}r)} 
 |f(y)| w_1 (y) w_1(y)^{-1} dy\\
	& \leq \|f\|_{L^{1,w_1}(B(a,2^{k+1}r))} \|w_1^{-1}\|_{L^\infty(B(a,2^{k+1}r))}.
\end{aligned}
$$
For $x \in B(a,r)$ we obtain
$$
\begin{aligned}
	|S f_2 (x)| & \leq C \sum_{k=1}^\infty \frac{1}{|B(a,2^{k+1}r)|} 
 \|f\|_{L^{1,w_1}(B(a,2^{k+1}r))} \|w_1^{-1}\|_{L^\infty(B(a,2^{k+1}r))} \\
	& = C \sum_{k=1}^\infty \left(2^{k+1}r\right)^{-n} \|f\|_{L^{1,w_1}(B(a,2^{k+1}r))} 
 \|w_1^{-1}\|_{L^\infty (B(a,2^{k+1}r))} \\
	& = C \sum_{k=1}^\infty \int_{2^{k+1}r}^{2^{k+2}r} \frac{dt}{t^{n + 1}} \cdot 
 \|f\|_{L^{1,w_1}(B(a,2^{k+1}r))} \|w_1^{-1}\|_{L^\infty(B(a,2^{k+1}r))} \\
	& = C \sum_{k=1}^\infty \int_{2^{k+1}r}^{2^{k+2}r} \|f\|_{L^{1,w_1}(B(a,t))} 
 \|w_1^{-1}\|_{L^\infty(B(a,t))}  \frac{dt}{t^{n + 1}}\\
	& \leq C \int_{2r}^\infty \|f\|_{L^{1,w_1}(B(a,t))} w_2(B(a,t))^{-1} \frac{dt}{t},
\end{aligned}
$$
and so
$$
	\|S f_2\|_{WL^{1,w_2}(B(a,r))} \leq \|S_\alpha f_2\|_{L^{1,w_2}(B(a,r))}\leq 
 C w_2(B(a,r)) \int_{2r}^\infty \|f\|_{L^{1,w_1}(B(a,t))} w_2(B(a,t))^{-1} \frac{dt}{t}.
$$
Therefore,
$$
	\|S f\|_{L^{p,w_2}(B(a,r))} \leq C_1 w_2(B(a,r))^\frac1p \int_r^\infty 
 w_2(B(a,t))^{-\frac1p} \|f\|_{L^{p,w_1}B(a,t)} \frac{dt}{t},
$$
for $1<p<\infty,$ and
$$
	\|S f\|_{WL^{1,w_2}(B(a,r))} \leq C_2 w_2(B(a,r)) \int_r^\infty w_2(B(a,t))^{-1} 
 \|f\|_{L^{1,w_1}B(a,t)} \frac{dt}{t},
$$
as desired.
\end{proof}

\begin{prop}\label{est-comm-sub-calder}
Let $1 < p<\infty$, $(w_1, w_2) \in A_{p/s}$ and $w_2 \in A_{p/s}$ where $1\le s<p$, 
and $w_1 \in A_\infty$. If $S$ is bounded from $L^{p,w_1}$ to $L^{p,w_2}$, then there 
exists a constant $C>0$ such that
$$
\|[b, S]f\|_{L^{p,w_2}(B(a,r))} \leq C \|b\|_* w_2(B(a,r))^{\frac{1}{p}} \int_r^\infty w_2(B(a,t))^{-\frac{1}{p}} \|f\|_{L^{p,w_1}B(a,t)} \frac{dt}{t},
$$
for any $(a,r) \in \mathbb{R}^n \times \mathbb{R}^+$ and $f \in L_{\rm loc}^{p,w_1}$.
\end{prop}

\begin{proof}
Let $a\in \mathbb{R}^n$ and $r>0$, and we represent the function $f \in L_{\rm loc}^{p,w_1}$ 
as $f := f_1 + f_2$ where $f_1 :=f\cdot \mathcal{X}_{B(a,2r)}$. Since $[b, S]$ is bounded 
from $L^{p,w_1}$ to $L^{p,w_2}$, then
$$
\|[b, S] f_1\|_{L^{p,w_2}(B(a,r))} \leq \|[b, S] f_1\|_{L^{p,w_2}} \leq C \|b\|_*\|f_1\|_{L^{p,w_1}} = C \|b\|_*\|f\|_{L^{p,w_1}(B(a,2r))}
$$
and Lemma \ref{f-1} implies that
\begin{equation}\label{Comf-Calder-1}
\|[b, S] f_1\|_{L^{p,w_2}(B(a,r))} \leq C \|b\|_* w_2(B(a,r))^\frac1p \int_{2r}^\infty 
\|f\|_{L^{p,w_1}(B(a,t))} w_2(B(a,t))^{-\frac1p} \frac{dt}{t}.
\end{equation}
	
On the other hand, by (\ref{commutator}),
$$
|[b, S] f_2 (x)| \leq C |b(x) - b_{B(a,r)}| |S (f_2) (x)| + |S ((b_{B(a,r)}-b)(f_2)(x)|, 
\quad x \in B(a,r).
$$
	
As for the sublinear operator generated by Riesz potential, we write
$$
I_1 (x) := |b(x) - b_{B(a,r)}| |S (f_2) (x)|, \quad I_2 (x) := |S ((b_{B(a,r)}-b)(f_2))(x)|.
$$
As in Theorem \ref{est-sub-calder}, we obtain that
$$
I_1 (x) \leq C |b(x) - b_{B(a,r)}| \int_{2r}^{\infty} \|f\|_{L^{p,w_1}(B(a,t))} 
w_2(B(a,t))^{-\frac1p} \frac{dt}{t}.
$$
The assumption $w_2^s \in A_{p/s}$ implies $w_2 \in A_p$ and by Lemma \ref{BMO3}, it follows 
that
\begin{equation}\label{I1}
\|I_1\|_{L^{p,w_2}(B(a,r))} \leq C \|b\|_* w_2(B(a,r))^\frac1p \int_{2r}^\infty 
\|f\|_{L^{p,w_1}(B(a,t))} w_2(B(a,t))^{-\frac1p} \frac{dt}{t}
\end{equation}
	
By (\ref{sublinear}), we have the following inequality.
$$
I_2 (x) \leq C \sum_{k=1}^\infty (2^{k+1}r)^{-\frac{n}s + \alpha} \left(\int_{B(a,2^{k+1}r)} 
|(b(y) - b_{B(a,r)})f(y)|^s dy\right)^\frac1s, \quad x \in B(a,r).
$$
H\"older inequality then implies that
$$
\left(\int_{B(a,2^{k+1}r)} |(b(y) - b_{B(a,r)})f(y)|^sdy\right)^\frac1s \leq C 
\|f\|_{L^{p,w_1}(B(a,2^{k+1}r))} \left\| (b-b_{B(a,r)})w^{-\frac1p} 
\right\|_{L^{ps/(p-s)}(B(a,2^{k+1}r))}.
$$
	
Thus, by the assumption that $w_1 \in A_\infty$, Lemma \ref{BMO3}, and Lemma \ref{f-1},
$$
\begin{aligned}
I_2 (x) & \leq C \sum_{k=1}^\infty (2^{k+1}r)^{-\frac{n}s} \|f\|_{L^{p,w_1}(B(a,2^{k+1}r))} 
\left\| (b-b_{B(a,r)})w^{-\frac1p} \right\|_{L^{ps/(p-s)}(B(a,2^{k+1}r))} \\
	& = C \sum_{k=1}^\infty \int_{2^{k+1}r}^{2^{k+2}r} \frac{dt}{t^{n+1}} \|f\|_{L^{p,w_1}(B(a,2^{k+1}r))}
\|(b-b_{B(a,r)}) w_1^{-\frac1p}\|_{L^{\frac{ps}{p-s}}(B(a,2^{k+1}r))} \\
	& \leq C \int_{2r}^\infty \|f\|_{L^{p,w_1}(B(a,t))} \|(b-b_{B(a,r)})
w_1^{\frac1p}\|_{L^{\frac{ps}{p-s}}(B(a,t))} \frac{dt}{t^{n+1}}\\
	& \leq C \int_{2r}^\infty \|f\|_{L^{p,w_1} (B(a,t))} \left(\int_{B(a,t)} 
|b(y)-b_{B(a,r)}|^{\frac{ps}{p-s}} w_1(y)^{-\frac{s}{p-s}} dy\right)^{\frac{p-s}{ps}} 
\frac{dt}{t^{n+1}}\\
	& = C \int_{2r}^\infty \|f\|_{L^{p,w_1}(B(a,t))} \|b\|_* \left(1 + \ln \frac{t}r \right) \left(w_1^{-\frac{s}{p-s}}(B(a,t))\right)^{\frac{p-s}{ps}} \frac{dt}{t^{n+1}} \\
	& = C \|b\|_* \int_{2r}^{\infty} \left(1+ \ln \frac{t}r \right) \|f\|_{L^{p,w_1}(B(a,t))} \|w_1^{-1}\|_{L^{\frac{s}{p-s}}(B(a,t))} \frac{dt}{t^{-\alpha+n+1}} \\
	& \leq C \|b\|_* \int_{2r}^\infty \left(1+\ln \frac{t}r\right) \|f\|_{L^{p,w_1}(B(a,t))} w_2(B(a,t))^{-\frac1p} \frac{dt}{t}
\end{aligned}
$$
which follows that
\begin{equation} \label{I-2}
	\|I_2\|_{L^{p,w_2}(B(a,r))}\leq C w_2(B(a,r)) \int_{2r}^\infty \left(1+\ln \frac{t}r\right) \|f\|_{L^{p,w_1}(B(a,t))} w_2(B(a,t))^{-\frac1p} \frac{dt}{t}.
\end{equation}
By combining (\ref{I1}) and (\ref{I-2}), we then obtain
\begin{equation} \label{Comf-Calder-2}
\|[b,S]f_2\|_{L^{p,w_2}(B(a,r))} \leq C \|b\|_* \int_{2r}^\infty \left(1+\ln 
\frac{t}r\right) \|f\|_{L^{p,w_1} (B(a,t))} w_2(B(a,t))^{-\frac1p} \frac{dt}{t}.
\end{equation}
	
Finally, (\ref{Comf-Calder-1}) and (\ref{Comf-Calder-2}) imply that
$$
\|[b,S]f\|_{L^{p,w_2}(B(a,r))} \leq C \|b\|_* \int_{2r}^\infty \left(1+\ln \frac{t}
{r}\right) \|f\|_{L^{p,w_1}(B(a,t))} w_2(B(a,t))^{-\frac1p} \frac{dt}{t}
$$
which proves the proposition.
\end{proof}

\begin{prop}\label{weak-est-sub-calder}
Let $(w_1, w_2) \in A_1$, $w_2 \in A_1$, and $w_1 \in A_\infty$. If $S$ satisfies 
(\ref{weak-commutator-calderon}), then there exists a constant $C>0$ such that for any 
$\sigma>0,$ $(a,r) \in \mathbb{R}^n \times \mathbb{R}^+$, and $f$,
$$
\|[b,S](f)\|_{WL^{1,w_2}} \leq C \sigma\|b\|_* w_2(B(a,r)) \int_{2r}^\infty \left(1 + 
\ln \frac{t}{r}\right) \left\| \Phi\left(\frac{|f|}{\sigma}\right) \right\|_{L \log 
L^{w_1}(B(a,t))} \frac{w_1B(a,t)}{w_2(B(a,t))}\frac{dt}{t}.
$$
\end{prop}

\begin{proof}
Let $a\in \mathbb{R}^n$ and $r>0$, and we represent the function $f$ as $f := f_1 + f_2$ 
where $f_1 := f\cdot \mathcal{X}_{B(a,2r)}$. We then have that
$$
\begin{aligned}
	& w\left(\left\{ x\in B(a,r) : |[b,S](f)(x)|>\sigma \right\}\right) \leq \\
	& w\left(\left\{ x\in B(a,r) : |[b,S](f_1)(x)|>\frac{\sigma}{2} \right\}\right) + 
 w\left(\left\{ x\in B(a,r) : |[b,S](f_2)(x)|>\frac{\sigma}{2} \right\}\right) := M_1 + M_2
\end{aligned}
$$
Since $S$ satisfies (\ref*{weak-commutator-calderon}), by (\ref{GenHolder}) and Lemma 
\ref{f-1} we have that
$$
\begin{aligned}
	& M_1 \leq C \int_{\mathbb{R}^n} \Phi \left(\frac{2|f_1(y)|}{\sigma}\right) w_1 (x) 
 dx \leq C \int_{\mathbb{R}^n} \Phi \left(\frac{2|f_1(y)|}{\sigma}\right) w_1 (x) dx \\
	& \leq C\int_{B(a,2r)} \Phi \left(\frac{2|f_1(y)|}{\sigma}\right) w_1 (x) dx
	\leq  C w_2(B(a,2r)) \int_{2r}^\infty \left\| \Phi\left(\frac{|f|}{\sigma}\right) 
 \right\|_{L^{1,w}(B(a,t))} w_2(B(a,t))^{-1} \frac{dt}{t} \\
	& \leq C w_2(B(a,r)) \int_{2r}^\infty \left\| \Phi\left(\frac{|f|}{\sigma}\right)
 \right\|_{L \log L^{w_1}(B(a,t))} \frac{w_1B(a,t)}{w_2(B(a,t))} \frac{dt}{t}.
\end{aligned}
$$
Hence,
\begin{equation}\label{weak-Comf-Calder-1}
	\|[b, S] f_1\|_{WL^{1,w_2}(B(a,r))} \leq C \sigma \|b\|_* w_2(B(a,r)) 
 \int_{2r}^\infty \left\| \Phi\left(\frac{|f|}{\sigma}\right) \right\|_{L \log 
 L^{w_1}(B(a,t))} \frac{w_1B(a,t)}{w_2(B(a,t))} \frac{dt}{t}.
\end{equation}
	
By (\ref{commutator}),
$$
\begin{aligned}
	& w_2\left(\left\{x\in B(a,r): |[b,S](f_2)(x)|> \frac{\sigma}{2}\right\}\right) \leq 
 w_2\left(\left\{x\in B(a,r): |b(x) - b_{B(a,r)}|\cdot|S(f_2)(x)|> 
 \frac{\sigma}{4}\right\}\right) \\
	& + w_2\left(\left\{x\in B(a,r): |S((b_{B(a,r)} - b)(f_2))(x)|> 
 \frac{\sigma}{4}\right\}\right) := M_1' + M_2'.
\end{aligned}
$$
Note that as in Theorem \ref{est-sub-calder}, by the fact that $t \leq \Phi(t)$ and 
(\ref{RelL-Log}), we have
$$
\begin{aligned}
	& \frac1\sigma |S(f_2)(x)| \leq C \sum_{k=1}^\infty \frac{1}{|2^{k+1}B(a,r)|} 
 \int_{2^{k+1}B(a,r)} \left|\frac{f(y)}{\sigma}\right| dy \leq C \int_{2r}^\infty 
 \left\|\frac{f}{\sigma}\right\|_{L^{1,w}(B(a,t))} w_2(B(a,t))^{-1} \frac{dt}{t} \\
	& \leq C \int_{2r}^\infty \left\|\Phi\left(\frac{f}{\sigma}\right)\right
 \|_{L^{1,w}(B(a,t))} w_2(B(a,t))^{-1} \frac{dt}{t} \leq C \int_{2r}^\infty \left\| 
 \Phi\left(\frac{|f|}{\sigma}\right) \right\|_{L \log L^{w_1}(B(a,t))} \frac{w_1B(a,t)}
 {w_2(B(a,t))} \frac{dt}{t}.
\end{aligned}
$$
For $M_1'$, the last inequality implies that
$$
\begin{aligned}
   M_1' & \leq \frac1\sigma \int_{B(a,r)} |b(x)-b_{B(a,r)}| \cdot |S(f_2)(x)|w_2(x) dx \\
	& \leq C \left(\int_{B(a,r)} |b(x)-b_{B(a,r)}| w_2(x) dx\right)  \int_{2r}^\infty 
 \left\| \Phi\left(\frac{|f|}{\sigma}\right) \right\|_{L \log L^{w_1}(B(a,t))} 
 \frac{w_1B(a,t)}{w_2(B(a,t))} \frac{dt}{t} \\
	& \leq C \|b\|_* w_2(B(a,r))\int_{2r}^\infty \left\| \Phi\left(\frac{|f|}
 {\sigma}\right) \right\|_{L \log L^{w_1}(B(a,t))} \frac{w_1B(a,t)}{w_2(B(a,t))} \frac{dt}{t}
\end{aligned}
$$
For $M_2'$, by the assumption that $(w_1, w_2) \in A_1$, Lemma \ref{Zhang}, and the fact 
$t \leq \Phi(t)$, we have by (\ref{GenHolder}) that
$$
\begin{aligned}
	& M_2' \leq C \frac1\sigma \int_{B(a,r)} |S((b_{B(a,r)}-b)f_2)(x)| w_2(x) dx \\
	&\leq w_2(B(a,r))
	\frac1\sigma  \sum_{k=1}^\infty \frac{1}{|2^{k+1}B(a,r)|} \int_{2^{k+1}B(a,r)} 
 |(b_{B(a,r)}-b(x))f(x)| dy \\
	& \leq C w_2(B(a,r)) \frac1\sigma \int_{2r}^\infty \|(b_{B(a,r)}-b)f
 \|_{L^{1,w_1}(B(a,t))} \|w_1^{-1}\|_{L^\infty (B(a,t))} \frac{dt}{t^{n+1}} \\
	& \leq C w_2(B(a,r)) \int_{2r}^\infty \left\|(b_{B(a,r)}-b)\frac{f}
 {\sigma}\right\|_{L^{1,w_1}(B(a,t))}  w_2(B(a,t))^{-1} \frac{dt}{t} \\
	& \leq C w_2(B(a,r)) \int_{2r}^\infty \left\|(b_{B(a,t)}-b)\frac{f}
 {\sigma}\right\|_{L^{1,w_1}(B(a,t))}  w_2(B(a,t))^{-1} \frac{dt}{t} \\
	& + C w_2(B(a,r)) \int_{2r}^\infty \left\|(b_{B(a,r)}-b_{B(a,t)})\frac{f}
 {\sigma}\right\|_{L^{1,w_1}(B(a,t))}  w_2(B(a,t))^{-1} \frac{dt}{t} \\
	& \leq C w_2(B(a,r))\int_{2r}^\infty \left\| \Phi\left(\frac{|f|}{\sigma}\right) 
 \right\|_{L \log L^{w_1}(B(a,t))} \left\|b-b_{B(a,t)}\right\|_{\exp L^{w_1}(B(a,t))} \frac{w_1B(a,t)}{w_2(B(a,t))} \frac{dt}{t}\\
	& + C w_2(B(a,r))\int_{2r}^\infty (b_{B(a,r)}-b_{B(a,t)}) \left\|\Phi\left(\frac{|f|}
 {\sigma}\right)\right\|_{L^{1,w_1}(B(a,t))} w_2(B(a,t))^{-1}\frac{dt}{t}\\
	& \leq C \|b\|_* w_2(B(a,r))\int_{2r}^\infty \left\| \Phi\left(\frac{|f|}
 {\sigma}\right) \right\|_{L \log L^{w_1}(B(a,t))} \frac{w_1B(a,t)}{w_2(B(a,t))} \frac{dt}{t} \\
	& + C \|b\|_* w_2(B(a,r))\int_{2r}^\infty \left(\ln \frac{t}{r}\right) \left\| 
 \Phi\left(\frac{|f|}{\sigma}\right) \right\|_{L \log L^{w_1}(B(a,t))} \frac{w_1B(a,t)}{w_2(B(a,t))}\frac{dt}{t} \\
	& \leq C \|b\|_* w_2(B(a,r)) \int_{2r}^\infty \left(1 + \ln \frac{t}{r}\right) 
 \left\| \Phi\left(\frac{|f|}{\sigma}\right) \right\|_{L \log L^{w_1}(B(a,t))} \frac{w_1B(a,t)}{w_2(B(a,t))}\frac{dt}{t}.
\end{aligned}
$$
Therefore,
$$
\begin{aligned}
	&w_2\left(\left\{x\in B(a,r): |[b,S](f_2)(x)|> \frac{\sigma}{2}\right\}\right) \\
	&\leq C \|b\|_* w_2(B(a,r)) \int_{2r}^\infty \left(1 + \ln \frac{t}{r}\right) 
 \left\| \Phi\left(\frac{|f|}{\sigma}\right) \right\|_{L \log L^{w_1}(B(a,t))} 
 \frac{w_1B(a,t)}{w_2(B(a,t))}\frac{dt}{t}
\end{aligned}
$$
and
\begin{equation}\label{weak-Comf-Calder-2}
\|[b,S](f_2)\|_{WL^{1,w_2}} \leq C \sigma\|b\|_* w_2(B(a,r)) \int_{2r}^\infty 
 \left(1 + \ln \frac{t}{r}\right) \left\| \Phi\left(\frac{|f|}{\sigma}\right) \right\|_{L 
 \log L^{w_1}(B(a,t))} \frac{w_1B(a,t)}{w_2(B(a,t))}\frac{dt}{t}
\end{equation}
By combining (\ref{weak-Comf-Calder-1}) and (\ref{weak-Comf-Calder-2}), we obtain that
$$
\|[b,S](f)\|_{WL^{1,w_2}} \leq C \sigma\|b\|_* w_2(B(a,r)) \int_{2r}^\infty \left(1 + \ln 
\frac{t}{r}\right) \left\| \Phi\left(\frac{|f|}{\sigma}\right) \right\|_{L \log 
L^{w_1}(B(a,t))} \frac{w_1B(a,t)}{w_2(B(a,t))}\frac{dt}{t}
$$
and this proves the proposition.
\end{proof}

Now, we shall prove Theorem \ref{Main-Calder}, Theorem \ref{Main-Comm-Calder}, Theorem 
\ref{Main-Calder-Weak}, Theorem \ref{Main-Calder-Mixed}, Theorem 
\ref{Main-Comm-Calder-Mixed}, and Theorem \ref{Main-Comm-Calder-Mixed-Weak}.

\begin{proof}[\textbf{Proof of Theorem \ref{Main-Calder}}]
By Lemma \ref{lemma-boundedness} the following inequality holds.
$$
\begin{aligned}
	\int_r^{\infty} w_2(B(a,t))^{-\frac{1}{p}} \|f\|_{L^{p,w_1}(B(a,t))} \frac{dt}{t}
	\leq \|f\|_{\mathcal{M}_{\psi_1}^{p,w_1}} \int_r^{\infty} \frac{\inf_{t\leq 
 l<\infty} \psi_1(a,l)w_1(B(a,l))^\frac1p}{w_2(B(a,t))^\frac1p} \frac{dt}{t}
\end{aligned}
$$
The last theorem, assumptions, and the last inequality imply that
$$
\begin{aligned}
\|S f\|_{\mathcal{M}_{\psi_2}^{p,w_2}} & = \sup_{(a,r) \in \mathbb{R}^n \times \mathbb{R}^+} 
\frac{1}{\psi_2 (a,r)} \frac{1}{w_2(B(a,r))^\frac1p} \|S f\|_{L^{p,w_2}(B(a,r))} \\
	& \leq C \sup_{(a,r) \in \mathbb{R}^n \times \mathbb{R}^+} \frac{1}{\psi_2(a,r)} 
 \int_{r}^\infty w_2(B(a,t))^{-\frac1p} \|f\|_{L^{p,w_1}(B(a,r))} \frac{dt}{t}\\
	& \leq C \|f\|_{\mathcal{M}_{\psi_1}^{p,w_1}} \sup_{(a,r)\in \mathbb{R}^n \times 
 \mathbb{R}^+} \frac{1}{\psi_2(a,r)} \int_r^\infty \frac{\inf_{t \leq l < \infty} 
 \psi_1 (a,l) w_1(B(a,l))^\frac1p}{w_2 (B(a,t))^\frac1p} \frac{dt}{t}  \leq C \|f\|_{\mathcal{M}_{\psi_1}^{p,w_1}}
\end{aligned}
$$
for $1<p<\infty$. On the other hand,
$$
\begin{aligned}
	\|S f \|_{W\mathcal{M}_{\psi_2}^{1,w_2}} & = \sup_{(a,r) \in \mathbb{R}^n \times 
 \mathbb{R}^+} \frac{1}{\psi_2(a,r)} \frac{1}{w_2(B(a,r))} \|T f\|_{WL^{1,w_2} (B(a,r))} \\
	& \leq C \sup_{(a,r) \in \mathbb{R}^n \times \mathbb{R}^+} \frac{1}{\psi_2(a,r)} \int_r^\infty w_2(B(a,t))^{-1} \|f\|_{L^{1,w}(B(a,t))} \frac{dt}{t}\\
	& \leq C \|f\|_{\mathcal{M}_{\psi_1}^{1,w}} \sup_{(a,r)\in \mathbb{R}^n \times \mathbb{R}^+} \frac{1}{\psi_2(a,r)} \int_r^\infty \frac{\inf_{t \leq l < \infty} \psi_1 (a,l) w_1(B(a,l))}{w_2 (B(a,r))} \frac{dt}{t} \leq C \|f\|_{\mathcal{M}_{\psi_1}^{1,w}}.
\end{aligned}
$$
These complete the proof of the theorem.
\end{proof}

\begin{proof}[\textbf{Proof of Theorem \ref{Main-Comm-Calder}}]
For $(a,r) \in \mathbb{R}^n \times \mathbb{R}^n$ the following inequality holds by Lemma \ref{lemma-boundedness}.
$$
\begin{aligned}
	\int_r^{\infty} \left(1+\ln \frac{t}r \right) w_2(B(a,t))^{-\frac{1}{p}} & 
 \|f\|_{L^{p,w_1}(B(a,t))} \frac{dt}{t} \\
	& \leq \|f\|_{\mathcal{M}_{\psi_1}^{p,w_1}} \int_r^{\infty} \left(1+\ln \frac{t}r 
 \right) \frac{\inf_{t\leq s<\infty} \psi_1(a,s)w_1(B(a,s))^\frac1p}{w_2(B(a,s))^\frac1p} \frac{dt}{t}.
\end{aligned}
$$
Theorem \ref{est-comm-sub-calder}, assumptions, and the last inequality then imply that
$$
\begin{aligned}
		& \|[b, S] f\|_{\mathcal{M}_{\psi_2}^{p,w_2}}
		\\ & = \sup_{(a,r) \in \mathbb{R}^n \times \mathbb{R}^+} \frac{1}{\psi_2 (a,r)} 
  \frac{1}{w_2(B(a,r))^\frac1p} \|[b, S] f\|_{L^{p,w_2}(B(a,r))} \\
		& \leq C \sup_{(a,r) \in \mathbb{R}^n \times \mathbb{R}^+} \frac{1}{\psi_2(a,r)} 
  \int_{r}^\infty \left(1+\ln \frac{t}r \right) w_2(B(a,t))^{-\frac1p} 
  \|f\|_{L^{p,w_1}(B(a,r))} \frac{dt}{t}\\
		& \leq C \|f\|_{\mathcal{M}_{\psi_1}^{p,w_1}} \sup_{(a,r)\in \mathbb{R}^n \times 
  \mathbb{R}^+} \frac{1}{\psi_2(a,r)} \int_r^\infty \left(1+\ln \frac{t}r \right) 
  \frac{\inf_{t \leq l < \infty} \psi_1 (a,l) w_1(B(a,l))^\frac1p}{w_2 (B(a,t))^\frac1p} 
  \frac{dt}{t} \leq C \|f\|_{\mathcal{M}_{\psi_1}^{p,w_1}}
\end{aligned}
	$$
	and this proves the theorem.
\end{proof}

\begin{proof}[\textbf{Proof of Theorem \ref{Main-Calder-Weak}}]
For $(a,r) \in \mathbb{R}^n \times \mathbb{R}^+$ the following inequality as in Lemma 
\ref{lemma-boundedness}.
	
$$
\begin{aligned}
	&\int_{2r}^\infty \left(1 + \ln \frac{t}{r}\right) \left\| \Phi\left(\frac{|f|}
 {\sigma}\right) \right\|_{L \log L^{w_1}(B(a,t))} \frac{w_1B(a,t)}{w_2(B(a,t))}\frac{dt}{t} \\
	& \leq \left\|\Phi \left(\frac{|f|}{\sigma}\right)\right\|_{\mathcal{M}_{\psi_1}^{L 
 \log L, w_1}} \int_r^{\infty} \left(1+\ln \frac{t}r \right)\inf_{t\leq s<\infty} 
 \psi_1(a,s) \frac{w_1(B(a,t))}{w_2(B(a,t))} \frac{dt}{t}.
\end{aligned}
$$
Theorem \ref{weak-est-sub-calder}, assumptions, and the last inequality then imply that
$$
\begin{aligned}
		& \|[b, S] f\|_{W\mathcal{M}_{\psi_2}^{1,w_2}}
		\\ & = \sup_{(a,r) \in \mathbb{R}^n \times \mathbb{R}^+} \frac{1}{\psi_2 (a,r)} 
  \frac{1}{w_2(B(a,r))^\frac1p} \|[b, T] f\|_{WL^{1,w_2}(B(a,r))} \\
		& \leq C \sigma \|b\|_* \sup_{(a,r) \in \mathbb{R}^n \times \mathbb{R}^+} 
  \frac{1}{\psi_2(a,r)} \int_{2r}^\infty \left(1 + \ln \frac{t}{r}\right) \left\| 
  \Phi\left(\frac{|f|}{\sigma}\right) \right\|_{L \log L^{w_1}(B(a,t))} 
  \frac{w_1B(a,t)}{w_2(B(a,t))}\frac{dt}{t}\\
		& \leq C \sigma \|b\|_* \left\|\Phi \left(\frac{|f|}
  {\sigma}\right)\right\|_{\mathcal{M}_{\psi_1}^{L \log L, w_1}} \sup_{(a,r) \in 
  \mathbb{R}^n \times \mathbb{R}^+} \frac1{\psi_2(a,r)} \int_r^{\infty} \left(1+\ln 
  \frac{t}r \right)\inf_{t\leq s<\infty} \psi_1(a,s) \frac{w_1(B(a,t))}{w_2(B(a,t))} 
  \frac{dt}{t} \\
		& \leq C \sigma \|b\|_* \left\|\Phi \left(\frac{|f|}{\sigma}\right)\right\|_{\mathcal{M}_{\psi_1}^{L \log L, w_1}}
\end{aligned}
$$
and this proves the theorem.
\end{proof}

\begin{proof}[\textbf{Proof of Theorem \ref{Main-Calder-Mixed}}]
$f(\cdot)$ may be replaced by $f(\cdot, t)$ in Theorem \ref{est-sub-calder} to obtain
$$
\|S f(\cdot, t)\|_{\mathcal{M}^{q,w_2}_{\psi_1}} \leq C \|b\|_* \|f(\cdot, t)
\|_{\mathcal{M}^{p,w_1}_{\psi_1}}, \quad f\in \mathcal{M}^{p,w_1}_{\psi_1}
$$
where $C>0$ is independent of $t$ and $1<p<\infty$. The theorem about the boundedness of $S$ 
on generalized weighted mixed-morrey spaces follows from Definition \ref{Mixed-wgwms} and 
Remark \ref{remark-mixed-morrey}.
\end{proof}
\begin{proof}[\textbf{Proof of Theorem \ref{Main-Comm-Calder-Mixed}}]
$f(\cdot)$ may be replaced by $f(\cdot, t)$ in Theorem \ref{est-comm-sub-calder} to obtain 
that
$$
\|[b, S]f(\cdot, t)\|_{\mathcal{M}^{q,w_2}_{\psi_1}} \leq C \|b\|_* \|f(\cdot, t)
\|_{\mathcal{M}^{p,w_1}_{\psi_1}},\quad f\in \mathcal{M}^{p,w_1}_{\psi_1}
$$
where $C>0$ is independent of $t$ and $1<p<\infty$. Then, Definition \ref{Mixed-wgwms} and 
Remark \ref{remark-mixed-morrey} imply the result.
\end{proof}

\begin{proof}[\textbf{Proof of Theorem \ref{Main-Comm-Calder-Mixed-Weak}}]
$f(\cdot)$ may be replaced by $f(\cdot, t)$ in Theorem \ref{weak-est-sub-calder} to obtain
$$
w_2(\{x \in \mathbb{R}^n: |[b,S](f)(x,t)|>\sigma\}) \leq C \int_{\mathbb{R}^n} \Phi 
\left(\frac{|f(x,t)|}{\sigma}\right) w_1 (x) dx
$$
where $C>0$ is independent of $t$. By Definition \ref{Mixed-wgwms} and Remark \ref{remark-mixed-morrey} we obtain the desired results about the boundedness of commutator generated by 
$S$ and $b$ on generalized weighted mixed-morrey spaces in the case $p=1$.
\end{proof}

\section{Applications to other Operators}

\subsection{Fractional integrals related to operators with Gaussian kernel bounds, 
fractional operators with rough kernels, and fractional maximal operators with rough kernels}

We provide the applications of our results in the section 3 for to some other operators particularly fractional integrals and fractional maximal operators, i.e. $L^{-\alpha/2}, S_{\Omega, \alpha}$, and $M_{\Omega, \alpha}$. First, we recall the definition and basic results of the operators. Let $L$ be a linear operator on $L^2$ which generates an analytic semigroup $\exp(-tL)$ with a kernel $p_t (x,y)$ satisfying
$$
|p_t(x,y)| \leq C \frac1{t^{n/2}} \exp \left(-c \frac{|x-y|^2}{t}\right); \quad x, y \in \mathbb{R}^n, t>0.
$$
For $0<\alpha<n,$ the fractional power $L^{-\alpha/n}$ of the operator $L$ is defined by
$$
L^{-\alpha/2}(f)(x) = \frac{1}{\Gamma(\alpha/2)} \int_0^\infty \exp (-tL) (f) (x) \frac{dt}{t^{-\alpha/2+1}}.
$$

\begin{theorem} \cite{Auscher} \label{fractional-integral}
$L^{-\alpha/2}$ and its commutator $[b,L^{-\alpha/2}]$ where $b \in BMO$ is bounded from $L^{p,w^p}$ to $L^{q,w^q}$ for $1<p< n /\alpha$ and $1/q=1/p-\alpha/n$.
\end{theorem}

Next, suppose that $\mathbb{S}^{n-1}$ is the unit sphere in $\mathbb{R}^n$ where $n\geq 2$ 
equipped with the normalized Lebesgue measure $d\sigma.$ Let $\Omega \in 
L^{s'}(\mathbb{S}^{n-1})$ be homogeneous of degree zero where $s'>1.$ For 
$0<\alpha<n$, the fractional integrals with rough kernel $\Omega$ is defined by
$$
S_{\Omega,\alpha} (f)(x) = \int_{\mathbb{R}^n} \frac{\Omega (y)}{|y|^{n-\alpha}} f(x-y)dy
$$
For $M_{\Omega, \alpha}$, suppose that $0<\alpha<n$ and $\Omega\in L^{s'} (\mathbb{S}^{n-1})$ 
is homogeneous of degree zero where $s'>1.$ The fractional maximal operator $M_{\Omega, 
\alpha}$ is defined by
$$
M_{\Omega,\alpha}(f)(x) = \sup_{r>0} \frac{1}{|B(x,r)|^{1-\alpha/n}} \int_{B(x,r)<r} 
|\Omega(x-y) f(y)| dy.
$$

\begin{theorem} \cite{Ding1998, Ding1999} \label{rough-kernel}
Let $0<\alpha<n, 1 < p< n/\alpha, 1/q=1/p-\alpha/n$, and $b\in BMO.$ Suppose that $\Omega 
\in L^{s'} (\mathbb{S})^{n-1}$ and $w^s\in A_{p/s,q/s}$ where $1\le s<p$. Then, 
$S_{\Omega,\alpha}$, $M_{\Omega,\alpha}$, $[b,S_{\Omega,\alpha}]$, and $[M_{\Omega, 
\alpha}]$ is bounded from $L^{p,w^p}$ to $L^{q,w^q}$.
\end{theorem}

Related to inequality (\ref{sublinear}), we have the following theorem.
\begin{lemma}\cite{Wang2017}\label{cor-riesz}
Suppose that $0<\alpha<n.$
\begin{enumerate}
	\item $L^{-\alpha/2}$ satisfies \ref{sublinear}.
	\item If $\Omega \in L^{s'} (\mathbb{S}^{n-1})$ is homogeneous of degree zero 
 where $1 < s' \leq \infty,$ then $S_{\Omega, \alpha}$ and $M_{\Omega, \alpha}$ satisfy 
 (\ref{sublinear}).
\end{enumerate}
\end{lemma}

The following theorems and corollaries generalize the results in \cite{Ramadana, Wang2017} 
and others by using Theorem \ref{fractional-integral} and Theorem \ref{rough-kernel}.

\begin{theorem}\label{Theo-Riesz-gwms-1}
Let $0 < \alpha < n, 1 \leq p < n/\alpha, 1/q = 1/p - \alpha/n$, $(w_1^s, w_2^s) \in 
A_{p/s,q/s}$, and $w_2^s \in A_{p/s, q/s}$, where $1\le s<p$ for $1<p<n/\alpha$ and $s=1$ 
for $p=1$.
\begin{enumerate}
	\item Suppose that the positive functions $\psi_1$ and $\psi_2$ satisfy 
(\ref{inequal-Riesz}). If $L^{-\alpha/2}$ is bounded from $L^{p,w_1^p}$ to $L^{q,w_2^q}$ 
for $1<p<n/\alpha$ and from $L^{1,w_1}$ to $WL^{q,w_2^q}$, then 
$L^{-\alpha/2}$ is bounded from $\mathcal{M}^{p,w_1^p}_{\psi_1}$ to $\mathcal{M}^{q,w_2^q}_{\psi_2}$ for $1<p<n/\alpha$ and from $\mathcal{M}^{1,w}_{\psi_1}$ to $W\mathcal{M}^{q,w_2^q}_{\psi_2}$.
	\item Suppose that the positive functions $\psi_1$ and $\psi_2$ satisfy 
(\ref{inequal-comm-riesz}) and $w_1\in A_\infty$. If $b\in BMO$ and $[b, L^{-\alpha/2}]$ is 
bounded from $L^{p,w_1^p}$ to $L^{q,w_2^q}$ for $1<p<n/\alpha$, then $[b,L^{-\alpha/2}]$ is 
bounded from $\mathcal{M}^{p,w_1^p}_{\psi_1}$ to $\mathcal{M}^{q,w_2^q}_{\psi_2}$ for $1<p<n/\alpha$.
\end{enumerate}
\end{theorem}

\begin{proof}
$L^{-\alpha/2}$ satisfies (\ref{sublinear}) by Lemma \ref{cor-riesz}. Plug $L^{-\alpha/2}$ 
onto $S_\alpha$ in Theorem \ref{Main-Riesz} and Theorem \ref{Main-Comm-Riesz} to obtain the 
results.
\end{proof}

\begin{theorem}\label{Theo-Riesz-Mixed-1}
Let $0 < \alpha < n, 1 \leq p < n/\alpha, 1/q = 1/p - \alpha/n$, $(w_1^s, w_2^s) \in 
A_{p/s,q/s}$, and $w_2^s \in A_{p/s, q/s}$, where $1\le s<p$ for $1<p<n/\alpha$ and $s=1$ 
for $p=1$. Suppose that $1<p_0<\infty$ and $\psi:\mathbb{R}^n\times \mathbb{R}^+ \to 
\mathbb{R}^+$.
\begin{enumerate}
	\item Suppose that the positive functions $\psi_1$ and $\psi_2$ satisfy 
(\ref{inequal-Riesz}). If $L^{-\alpha/2}$ is bounded from $L^{p,w_1^p}$ to $L^{q,w_2^q}$ for 
$1<p<n/\alpha$ and from $L^{1,w_1}$ to $WL^{q,w_2^q}$, then 
$L^{-\alpha/2}$ is bounded from $\mathcal{M}_\psi^{p_0,w}(0, T, 
\mathcal{M}^{p,w_1^p}_{\psi_1})$ to $\mathcal{M}_\psi^{p_0,w}(0, T, 
\mathcal{M}^{q,w_2^q}_{\psi_2})$ for $1<p<n/\alpha$ and 
from $\mathcal{M}_\psi^{p_0,w}(0, T, \mathcal{M}^{1,w}_{\psi_1})$ to 
$\mathcal{M}_\psi^{p_0,w} (0, T, W\mathcal{M}^{q,w^q}_{\psi_2})$.
	\item Suppose that the positive functions $\psi_1$ and $\psi_2$ satisfy 
(\ref{inequal-comm-riesz}) and $w_1\in A_\infty$. If $b\in BMO$ and $[b,L^{-\alpha/2}]$ is 
bounded from $L^{p,w_1^p}$ to $L^{q,w_2^q}$ for $1<p<n/\alpha$, 
then $[b,L^{-\alpha/2}]$ is bounded from $\mathcal{M}_\psi^{p_0,w} (0, T, 
\mathcal{M}^{p,w_1^p}_{\psi_1})$ to $\mathcal{M}_\psi^{p_0,w} (0, T, 
\mathcal{M}^{q,w_2^q}_{\psi_2})$ for $1<p<n/\alpha$.
\end{enumerate}
\end{theorem}

\begin{proof}
Just plug $L^{-\alpha/2}$ onto $S_\alpha$ in Theorem \ref{Main-Riesz-Mixed} and Theorem 
\ref{Main-Comm-Riesz-Mixed}.
\end{proof}

\begin{cor}\label{Cor-Riesz-gwms-1}
Let $0 < \alpha < n,\,1 <p<n/\alpha,\,1/q=1/p-\alpha/n$, and $w^s \in A_{p/s,q/s}$, where $1\le s<p$. Suppose that $\Omega\in L^{s'}(\mathbb{S}^{n-1})$, $1<p_0<\infty$ and $\psi:\mathbb{R}^n\times \mathbb{R}^+ \to \mathbb{R}^+$.
\begin{enumerate}
\item Suppose that the positive functions $\psi_1$ and $\psi_2$ satisfy 
(\ref{inequal-Riesz}). Then, $L^{-\alpha/2}$ is bounded from $\mathcal{M}^{p,w^p}_{\psi_1}$ 
to $\mathcal{M}^{q,w^q}_{\psi_2}$.
\item Suppose that the positive functions $\psi_1$ and $\psi_2$ satisfy 
(\ref{inequal-comm-riesz}). If $b\in BMO$, then $[b,L^{-\alpha/2}]$ is bounded from 
$\mathcal{M}^{p,w^p}_{\psi_1}$ to $\mathcal{M}^{q,w^q}_{\psi_2}$.
\end{enumerate}
\end{cor}

\begin{cor}\label{Cor-Riesz-Mixed-1}
Let $0 < \alpha < n,\ 1 <p<n/\alpha,\ 1/q=1/p-\alpha/n$, and $w^s \in A_{p/s,q/s}$,
where $1\le s<p$. Suppose that $1<p_0<\infty$ and $\psi:\mathbb{R}^n\times \mathbb{R}^+ \to \mathbb{R}^+$.
\begin{enumerate}
\item Suppose that the positive functions $\psi_1$ and $\psi_2$ satisfy 
(\ref{inequal-Riesz}). Then, $L^{-\alpha/2}$ is bounded from $\mathcal{M}_\psi^{p_0,w}(0, T,
\mathcal{M}^{p,w^p}_{\psi_1})$ to $\mathcal{M}_\psi^{p_0,w}(0, T, 
\mathcal{M}^{q,w^q}_{\psi_2})$.
\item Suppose that the positive functions $\psi_1$ and $\psi_2$ satisfy 
(\ref{inequal-comm-riesz}). If $b\in BMO$, then $[b,L^{-\alpha/2}]$ is bounded from 
$\mathcal{M}_\psi^{p_0,w}(0, T, \mathcal{M}^{p,w^p}_{\psi_1})$ to 
$\mathcal{M}_\psi^{p_0,w}(0, T, \mathcal{M}^{q,w^q}_{\psi_2})$.
\end{enumerate}
\end{cor}

\begin{theorem}\label{Theo-Riesz-gwms-2}
Let $0 < \alpha < n, 1 \le p< n/\alpha, 1/q = 1/p - \alpha/n$, $(w_1^s, w_2^s) \in 
A_{p/s,q/s}$, and $w_2^s \in A_{p/s, q/s}$, where $1\le s<p$ for $1<p<n/\alpha$ and $s=1$
for $p=1$. Suppose that $\Omega\in L^{s'}(\mathbb{S}^{n-1})$, $1<p_0<\infty$ and 
$\psi:\mathbb{R}^n\times \mathbb{R}^+ \to \mathbb{R}^+$.
\begin{enumerate}
	\item Suppose that the positive functions $psi_1$ and $\psi_2$ satisfy 
 (\ref{inequal-Riesz}). If $S_{\Omega,\alpha}$ is bounded from $L^{p,w_1^p}$ to 
 $L^{q,w_2^q}$ for $1<p<n/\alpha$ and from $L^{1,w_1}$ to $WL^{q,w^q}$, then 
 $S_{\Omega,\alpha}$ is bounded from $\mathcal{M}^{p,w_1^p}_{\psi_1}$ to 
 $\mathcal{M}^{q,w_2^q}_{\psi_2}$ for $1<p<n/\alpha$ and from $\mathcal{M}^{1,w}_{\psi_1}$ 
 to $W\mathcal{M}^{q,w^q}_{\psi_2}$.
	\item Suppose that the positive functions $\psi_1$ and $\psi_2$ satisfy 
 (\ref{inequal-comm-riesz}) and $w_1\in A_\infty$.  If $b\in BMO$ and 
 $[b,S_{\Omega,\alpha}]$ is bounded from $L^{p,w_1^p}$ to $L^{q,w_2^q}$ for $1<p<n/\alpha$, 
 then $[b,S_{\Omega,\alpha}]$ is bounded from $\mathcal{M}^{p,w_1^p}_{\psi_1}$ to 
 $\mathcal{M}^{q,w_2^q}_{\psi_2}$ for $1<p<n/\alpha$.
\end{enumerate}
\end{theorem}

\begin{proof}
$S_{\Omega,\alpha}$ satisfies (\ref{sublinear}) by Lemma \ref{cor-riesz} point (2). The 
results immediately are obtained by plugging $S_{\Omega,\alpha}$ onto $S_\alpha$ in Theorem 
\ref{Main-Riesz} and Theorem \ref{Main-Comm-Riesz}.
\end{proof}

\begin{theorem}\label{Theo-Riesz-Mixed-2}
Let $0 < \alpha < n, 1 \leq p < n/\alpha, 1/q = 1/p - \alpha/n$, $(w_1^s, w_2^s) \in 
A_{p/s,q/s}$, and $w_2^s \in A_{p/s, q/s}$, where $1\le s<p$ for $1<p<n/\alpha$ and $s=1$ 
for $p=1$. Suppose that $\Omega\in L^{s'}(\mathbb{S}^{n-1})$, $1<p_0<\infty$ and 
$\psi:\mathbb{R}^n\times \mathbb{R}^+ \to \mathbb{R}^+$. \begin{enumerate}
	\item Suppose that the positive functions $psi_1$ and $\psi_2$ satisfy 
 (\ref{inequal-Riesz}).  If $S_{\Omega,\alpha}$ is bounded from $L^{p,w_1^p}$ to 
 $L^{q,w_2^q}$ for $1<p<\infty$ and from $L^{1,w_1}$ to $WL^{q,w^q}$, then 
 $S_{\Omega,\alpha}$ is bounded from $\mathcal{M}_\psi^{p_0,w}(0, T, 
 \mathcal{M}^{p,w_1^p}_{\psi_1})$ to $\mathcal{M}_\psi^{p_0,w}(0, T, 
 \mathcal{M}^{q,w_2^q}_{\psi_2})$ for $1<p<n/\alpha$ and from $\mathcal{M}_\psi^{p_0,w}(0, 
 T, \mathcal{M}^{1,w}_{\psi_1})$ to $\mathcal{M}_\psi^{p_0,w} (0, T, 
 W\mathcal{M}^{q,w^q}_{\psi_2})$.
	\item Suppose that the positive functions $\psi_1$ and $\psi_2$ satisfy 
 (\ref{inequal-comm-riesz}) and $w_1\in A_\infty$.  If $b\in BMO$ and 
 $[b,S_{\Omega,\alpha}]$ is bounded from $L^{p,w_1^p}$ to $L^{q,w_2^q}$ for $1<p<n/\alpha$, 
 then $[b,S_{\Omega,\alpha}]$ is bounded from $\mathcal{M}_\psi^{p_0,w} (0, T, 
 \mathcal{M}^{p,w_1^p}_{\psi_1})$ to $\mathcal{M}_\psi^{p_0,w} (0, T, 
 \mathcal{M}^{q,w_2^q}_{\psi_2})$ for $1<p<n/\alpha$.
\end{enumerate}
\end{theorem}

\begin{proof}
Just plug $S_{\Omega,\alpha}$ onto $S_\alpha$ in Theorem \ref{Main-Riesz-Mixed} and Theorem 
\ref{Main-Comm-Riesz-Mixed}.
\end{proof}

\begin{cor}\label{Cor-Riesz-gwms-2}
Let $0 < \alpha < n, 1 <p<n/\alpha, 1/q = 1/p - \alpha/n$, and $w^s \in A_{p/s,q/s}$, where 
$1\le s<p$. Suppose that $\Omega\in L^{s'}(\mathbb{S}^{n-1})$, $1<p_0<\infty$ and 
$\psi:\mathbb{R}^n\times \mathbb{R}^+ \to \mathbb{R}^+$.
\begin{enumerate}
	\item Suppose that the positive functions $\psi_1$ and $\psi_2$ satisfy 
(\ref{inequal-Riesz}). Then, $S_{\Omega,\alpha}$ is bounded from 
$\mathcal{M}^{p,w_1^p}_{\psi_1}$ to $\mathcal{M}^{q,w_2^q}_{\psi_2}$.
	\item Suppose that the positive functions $\psi_1$ and $\psi_2$ satisfy (\ref{inequal-comm-riesz}). If $b\in BMO$, then $[b,S_{\Omega,\alpha}]$ is bounded from $\mathcal{M}^{p,w_1^p}_{\psi_1}$ to $\mathcal{M}^{q,w_2^q}_{\psi_2}$.
	\end{enumerate}
\end{cor}

\begin{cor}\label{Cor-Riesz-Mixed-2}
Let $0 < \alpha < n, 1 < p<n/\alpha, 1/q = 1/p - \alpha/n$, and $w^s \in A_{p/s,q/s}$, where 
$1\le s<p$. Suppose that $\Omega\in L^{s'}(\mathbb{S}^{n-1})$, $1<p_0<\infty$ and 
$\psi:\mathbb{R}^n\times \mathbb{R}^+ \to \mathbb{R}^+$. \begin{enumerate}
	\item Suppose that the positive functions $\psi_1$ and $\psi_2$ satisfy 
(\ref{inequal-Riesz}). Then, $S_{\Omega,\alpha}$ is bounded from $\mathcal{M}_\psi^{p_0,w}
(0, T, \mathcal{M}^{p,w_1^p}_{\psi_1})$ to $\mathcal{M}_\psi^{p_0,w}(0, T, 
\mathcal{M}^{q,w_2^q}_{\psi_2})$.
	\item Suppose that the positive functions $\psi_1$ and $\psi_2$ satisfy 
(\ref{inequal-comm-riesz}). Then, $[b,S_{\Omega,\alpha}]$ is bounded from 
$\mathcal{M}_\psi^{p_0,w} (0, T, \mathcal{M}^{p,w_1^p}_{\psi_1})$ to 
$\mathcal{M}_\psi^{p_0,w} (0, T, \mathcal{M}^{q,w_2^q}_{\psi_2})$.
\end{enumerate}
\end{cor}

\begin{theorem}\label{Theo-Riesz-gwms-3}
Let $0 < \alpha < n, 1 \leq p < n/\alpha, 1/q = 1/p - \alpha/n$, $(w_1^s, w_2^s) 
\in A_{p/s,q/s}$, and $w_2^s \in A_{p/s, q/s}$, where $1\le s<p$ for $1<p<n/\alpha$ and
$s=1$ for $p=1$. Suppose that $\Omega\in L^{s'}(\mathbb{S}^{n-1})$, $1<p_0<\infty$ and $\psi:\mathbb{R}^n\times \mathbb{R}^+ \to \mathbb{R}^+$.
\begin{enumerate}
	\item Suppose that the positive functions $\psi_1$ and $\psi_2$ satisfy 
(\ref{inequal-Riesz}). If $M_{\Omega,\alpha}$ is bounded from $L^{p,w_1^p}$ to 
$L^{q,w_2^q}$ for $1<p<n/\alpha$ and from $L^{1,w_1}$ to $WL^{q,w^q}$, then 
$M_{\Omega,\alpha}$ is bounded from $\mathcal{M}^{p,w_1^p}_{\psi_1}$ to 
$\mathcal{M}^{q,w_2^q}_{\psi_2}$ for $1<p<n/\alpha$ and from $\mathcal{M}^{1,w}_{\psi_1}$ to 
$W\mathcal{M}^{q,w^q}_{\psi_2}$.
	\item Suppose that the positive functions $\psi_1$ and $\psi_2$ satisfy 
(\ref{inequal-comm-riesz}) and $w_1\in A_\infty$. If $b\in BMO$ and $[b,M_{\Omega,\alpha}]$ 
is bounded from $L^{p,w_1^p}$ to $L^{q,w_2^q}$ for $1<p<n/\alpha$, then 
$[b,M_{\Omega,\alpha}]$ is bounded from $\mathcal{M}^{p,w_1^p}_{\psi_1}$ to 
$\mathcal{M}^{q,w_2^q}_{\psi_2}$ for $1<p<n/\alpha$.
\end{enumerate}
\end{theorem}

\begin{proof}
By Lemma \ref{cor-riesz}, $M_{\Omega,\alpha}$ satisfies (\ref{sublinear}). We plug 
$M_{\Omega,\alpha}$ onto $S_\alpha$ in Theorem \ref{Main-Riesz} and Theorem 
\ref{Main-Comm-Riesz} to obtain the the boundedness properties.
\end{proof}

\begin{theorem}\label{Theo-Riesz-Mixed-3}
Let $0 < \alpha < n, 1 \leq p < n/\alpha, 1/q = 1/p - \alpha/n$, $(w_1^s, w_2^s) \in 
A_{p/s,q/s}$, and $w_2^s \in A_{p/s, q/s}$, where $1\le s<p$ for $1<p<n/\alpha$ and $s=1$ 
for $p=1$. Suppose that $\Omega\in L^{s'}(\mathbb{S}^{n-1})$, $1<p_0<\infty$ and $\psi:\mathbb{R}^n\times \mathbb{R}^+ \to \mathbb{R}^+$.
\begin{enumerate}
	\item Suppose that the positive functions $\psi_1$ and $\psi_2$ satisfy 
(\ref{inequal-Riesz}). If $M_{\Omega,\alpha}$ is bounded from $L^{p,w_1^p}$ to 
$L^{q,w_2^q}$ for $1<p<n/\alpha$ and from $L^{1,w_1}$ to $WL^{q,w^q}$, then 
$M_{\Omega,\alpha}$ is bounded from $\mathcal{M}_\psi^{p_0,w}(0, T, 
\mathcal{M}^{p,w_1^p}_{\psi_1})$ to $\mathcal{M}_\psi^{p_0,w}(0, T, 
\mathcal{M}^{q,w_2^q}_{\psi_2})$ for $1<p<n/\alpha$ and from $\mathcal{M}_\psi^{p_0,w}(0, T, 
\mathcal{M}^{1,w}_{\psi_1})$ to $\mathcal{M}_\psi^{p_0,w} (0, T, 
W\mathcal{M}^{q,w^q}_{\psi_2})$.
	\item Suppose that the positive functions $\psi_1$ and $\psi_2$ satisfy 
(\ref{inequal-comm-riesz}) and $w_1\in A_\infty$. If $b\in BMO$ and $[b,M_{\Omega,\alpha}]$
is bounded from $L^{p,w_1^p}$ to $L^{q,w_2^q}$ for $1<p<n/\alpha$, then 
$[b,M_{\Omega,\alpha}]$ is bounded from $\mathcal{M}_\psi^{p_0,w} (0, T, 
\mathcal{M}^{p,w_1^p}_{\psi_1})$ to $\mathcal{M}_\psi^{p_0,w} (0, T, \mathcal{M}^{q,w_2^q}_{\psi_2})$ for $1<p<n/\alpha$.
\end{enumerate}
\end{theorem}

\begin{proof}
Put $M_{\Omega,\alpha} = S_\alpha$ in Theorem \ref{Main-Riesz-Mixed} and Theorem 
\ref{Main-Comm-Riesz-Mixed}.
\end{proof}

\begin{cor}\label{Cor-Riesz-gwms-3}
Let $0 < \alpha < n, 1 < p<n/\alpha, 1/q = 1/p - \alpha/n$, and $w^s \in A_{p/s,q/s}$, where 
$1\le s<p$. Suppose that $\Omega\in L^{s'}(\mathbb{S}^{n-1})$, $1<p_0<\infty$ and 
$\psi:\mathbb{R}^n\times \mathbb{R}^+ \to \mathbb{R}^+$.
\begin{enumerate}
	\item Suppose that the positive functions $\psi_1$ and $\psi_2$ satisfy 
(\ref{inequal-Riesz}). Then, $M_{\Omega,\alpha}$ is bounded from 
$\mathcal{M}^{p,w_1^p}_{\psi_1}$ to $\mathcal{M}^{q,w_2^q}_{\psi_2}$.
	\item Suppose that the positive functions $\psi_1$ and $\psi_2$ satisfy 
(\ref{inequal-comm-riesz}). If $b\in BMO$, then $[b,M_{\Omega,\alpha}]$ is bounded from $\mathcal{M}^{p,w_1^p}_{\psi_1}$ to $\mathcal{M}^{q,w_2^q}_{\psi_2}$.
\end{enumerate}
\end{cor}

\begin{cor}\label{Cor-Riesz-Mixed-3}
Let $0 < \alpha < n, 1 < p < n/\alpha, 1/q = 1/p - \alpha/n$, and $w^s \in A_{p/s,q/s}$, 
where $1\le s<p$. Suppose that $\Omega\in L^{s'}(\mathbb{S}^{n-1})$, $1<p_0<\infty$ and 
$\psi:\mathbb{R}^n\times \mathbb{R}^+ \to \mathbb{R}^+$.
\begin{enumerate}
	\item Suppose that the positive functions $\psi_1$ and $\psi_2$ satisfy 
(\ref{inequal-Riesz}). Then, $M_{\Omega,\alpha}$ is bounded from 
$\mathcal{M}_\psi^{p_0,w}(0, T, \mathcal{M}^{p,w_1^p}_{\psi_1})$ to 
$\mathcal{M}_\psi^{p_0,w}(0, T, \mathcal{M}^{q,w_2^q}_{\psi_2})$.
	\item Suppose that the positive functions $\psi_1$ and $\psi_2$ satisfy 
(\ref{inequal-comm-riesz}). If $b\in BMO$, then $[b,M_{\Omega,\alpha}]$ is bounded from 
$\mathcal{M}_\psi^{p_0,w}(0, T, \mathcal{M}^{p,w_1^p}_{\psi_1})$ to 
$\mathcal{M}_\psi^{p_0,w}(0, T, \mathcal{M}^{q,w_2^q}_{\psi_2})$.
\end{enumerate}
\end{cor}

\medskip

\subsection{Singular integral operators}

Let $\Omega$ be as in Subsection 5.2. Suppose that $S_\Omega$ is a sublinear operator such 
that for any $f\in L^1$ with compact support and $x \notin {\rm supp}\,(f)$
\begin{equation}\label{sublinear-1}
	|S_\Omega(f)(x)| \leq C \int_{\mathbb{R}^n} \frac{|\Omega(x-y)|}{|x-y|^n}|f(y)| dy.
\end{equation}

In fact, $S_\Omega$ satisfies (\ref{sublinear}) as in the following lemma.

\begin{lemma}\label{Lemma-Cald}
$S_\Omega$ satisfies (\ref{sublinear}) for $\Omega \in L^{s'} (\mathbb{S}^{n-1})$ where $1<s'\leq \infty.$
\end{lemma}
\begin{proof}
By H\"older inequality we deduce that
$$
\begin{aligned}
		&|S_\Omega (f\cdot\mathcal{X}_{B(a,2r)^c}) (x)| \leq C \int_{B(a,2r)^c} \frac{|\Omega(x-y)|}{|x-y|^n} |f(y)| dy \\
		& = C \sum_{k=1}^\infty \int_{B(a,2^{k+1}r)\setminus B(a,2^kr)} \frac{|\Omega(x-y)|}{|x-y|^n} |f(y)| dy \\
		& \leq \sum_{k=1}^\infty \int_{B(a,2^{k+1}r)\setminus B(a,2^kr)} \frac{|\Omega(x-y)|}{|x-y|^n} |f(y)| dy \\
		& \leq C \sum_{k=1}^\infty \left(\int_{B(a,2^{k+1}r)\setminus B(a,2^kr)} |\Omega(x-y)|^{s'} dy \right)^{\frac1{s'}} \left(\int_{B(a,2^{k+1}r)\setminus B(a,2^kr)} \frac{|f(y)|^s}{|x-y|^{ns}} dy \right)^{\frac1s}.
\end{aligned}
$$
We note from that
$$
\begin{aligned}
	\left(\int_{B(a,2^{k+1}r)\setminus B(a,2^kr)} |\Omega(x-y)|^{s'} dy 
 \right)^{\frac1{s'}} & \leq C \|\Omega\|_{L^{s'}(\mathbb{S}^{n-1})}|B(0,2^{k+1}r+
 |x-a|)|^{\frac{1}{s'}} \\
	& \leq C \|\Omega\|_{L^{s'}(\mathbb{S}^{n-1})}|B(a,2^{k+1}r)|^{\frac{1}{s'}}
\end{aligned}
$$
where $C>0$ is dependent of $x,k,$ and $a$ (see \cite{Guliyev2016} for details). We also 
note that $|x-y| \sim |y-a|$ for $x\in B(a,r)$ and $y\in B(a,2r)^c.$ Hence,
$$
	\left(\int_{B(a,2^{k+1}r)\setminus B(a,2^kr)} \frac{|f(y)|^s}{|x-y|^{ns}} dy 
 \right)^{\frac1s} \leq C \frac{1}{|B(a,2^{k+1}r)|} \left(\int_{B(a,2^{k+1}r)\setminus 
 B(a,2^kr)} |f(y)|^{s} dy \right)^{\frac1s}.
$$
Consequently,
$$
	|S_\Omega (f\cdot\mathcal{X}_{B(a,2r)^c}) (x)| \leq C \frac{1}{|B(a,2^{k+1}r)|^{1-
 \frac{1}{s'}}} \left(\int_{B(a,2^{k+1}r)\setminus B(a,2^kr)} |f(y)|^{s} dy \right)^{\frac1s}
$$
which implies that $S_\Omega$ satisfies (\ref{sublinear}), and the theorem is proved.
\end{proof}

By the lemma, we have the following theorem concerning the boundedness of sublinear operator with rough kernel and its commutator on generalized weighted Morrey spaces and generalized weighted mixed-Morrey spaces with different weights. The following theorem generalizes many other knwon results, for example \cite{Guliyev2016, Hamzayev, Karaman, Ragusa2017, Ramadana}.

\begin{theorem}
Let $1 \leq p < \infty, \Omega \in L^{s'}(\mathbb{S}^{n-1})$, $(w_1^s, w_2^s), w_2^s \in A_{p/s}$ 
where $1\le s<p$ for $p>1$ and $s=1$ for $p=1$, and $S_\Omega$ satisfies (\ref{sublinear-1}).
	\begin{enumerate}
		\item Suppose that the two positive functions $\psi_1$ and $\psi_2$ on $\mathbb{R}^n \times \mathbb{R}^+$ satisfy (\ref{inequal-calder}). If $S_\Omega$ is bounded from $L^{p,w_1}$ to $L^{p,w_2}$ for $1<p<\infty$ and bounded from $L^{1,w_1}$ to $WL^{1,w_2}$, then $S_\Omega$ is bounded from $\mathcal{M}^{p,w_1}_{\psi_1}$ to $\mathcal{M}^{p,w_2}_{\psi_2}$ for $1<p<\infty$ and bounded from $\mathcal{M}^{1,w_1}_{\psi_1}$ to $W\mathcal{M}^{1,w_2}_{\psi_2}$.
		\item Suppose that the two positive functions $\psi_1$ and $\psi_2$ on $\mathbb{R}^n \times \mathbb{R}^+$ satisfy (\ref{inequal-comm-calder}) and $w_1\in A_\infty$. If $[b, S_\Omega]$ is bounded from $L^{p,w_1}$ to $L^{p,w_2}$ for $1<p<\infty$,  then $[b, S_\Omega]$ is bounded from $\mathcal{M}^{p,w_1}_{\psi_1}$ to $\mathcal{M}^{p,w_2}_{\psi_2}$ for $1<p<\infty$.
		\item Suppose that the two positive functions $\psi_1$ and $\psi_2$ on $\mathbb{R}^n \times \mathbb{R}^+$ satisfy (\ref{weak-inequal-comm-calder}) and $w_1\in A_\infty$. If $S_\Omega$ satisfies (\ref{weak-commutator-calderon}),  then there exists a constant $C>0$ such that
		$$
		\|[b,S_\Omega](f)\|_{W\mathcal{M}_{\psi_2}^{1,w_2}} \leq C \|b\|_* \sup_{\sigma>0} \sigma \left\|\Phi \left(\frac{|f|}{\sigma}\right)\right\|_{\mathcal{M}_{\psi_1}^{L \log L, w_1}}
		$$
	\end{enumerate}
\end{theorem}
\begin{proof}
	$S_\Omega$ satisfies (\ref{sublinear}) by Lemma \ref{Lemma-Cald}. Plug $S_\Omega$ onto $S$ in Theorem \ref{Main-Calder}, Theorem \ref{Main-Comm-Calder}, and Theorem \ref{Main-Calder-Weak} to obtain the desired results.
\end{proof}

\begin{theorem}
Let $\psi$ be a positive function on $\mathbb{R}^n \times \mathbb{R}^+$, $0<p_0<\infty$, 
$1 \leq p < \infty$, $\Omega\in L^{s'}(\mathbb{S}^{n-1}),$ and $(w_1^s, w_2^s), w_2^s \in A_{p/s}$, where $1\le s<p$ for $p>1$ and $s=1$ for $p=1$. Suppose that $1<p_0<\infty$, 
$\psi:\mathbb{R}^n\times\mathbb{R}$, and $S_\Omega$ satisfies (\ref{sublinear-1}).
\begin{enumerate}
	\item Suppose that the two functions $\psi_1$ and $\psi_2$ satisfy 
(\ref{inequal-calder}). If $S_\Omega$ is bounded from $L^{p,w_1}$ to $L^{p,w_2}$ for 
$1<p<\infty$ and bounded from $L^{1,w_1}$ to $WL^{1,w_2}$, then $S_\Omega$ is bounded from 
the space $\mathcal{M}_\psi^{p_0,w}(0, T, \mathcal{M}^{p,w_1}_{\psi_1})$ to the space 
$\mathcal{M}_\psi^{p_0,w}(0, T, \mathcal{M}^{p,w_2}_{\psi_2})$ for $1<p<\infty$ and bounded 
from $\mathcal{M}_\psi^{p_0,w}(0, T, \mathcal{M}^{1,w_1}_{\psi_1})$ to the space 
$\mathcal{M}_\psi^{p_0,w}(0, T, W\mathcal{M}^{1,w_2}_{\psi_2})$.
	\item Suppose that the two functions $\psi_1$ and $\psi_2$ satisfy 
(\ref{inequal-comm-calder}) and $w_1\in A_\infty$. If $S_\Omega$ is bounded from $L^{p,w_1}$ 
to $L^{p,w_2}$ for $1<p<\infty$,  then $[b, S_\Omega]$ is bounded from 
$\mathcal{M}_\psi^{p_0,w}(0, T, \mathcal{M}^{p,w_1}_{\psi_1})$ to 
$\mathcal{M}_\psi^{p_0,w}(0, T, \mathcal{M}^{p,w_2}_{\psi_2})$ for $1<p<\infty$.
	\item Suppose that the two functions $\psi_1$ and $\psi_2$ satisfy 
 (\ref{weak-inequal-comm-calder}) and $w_1\in A_\infty$. If $S_\Omega$ satisfies 
 (\ref{weak-commutator-calderon}),  then there exists a constant $C>0$ such that for any 
 suitable function $f,$
$$
	\|[b,S_\Omega](f)\|_{\mathcal{M}_\psi^{p_0,w}(0, T, W\mathcal{M}_{\psi_2}^{1,w_2})} 
 \leq C \|b\|_* \sup_{\sigma>0} \sigma \left\|\Phi \left(\frac{|f|}
 {\sigma}\right)\right\|_{\mathcal{M}_\psi^{p_0,w}(0, T, \mathcal{M}_{\psi_1}^{L \log L, 
 w_1})}.
$$
\end{enumerate}
\end{theorem}

\begin{proof}
Just plug $S_\Omega$ onto $S$ in Theorem \ref{Main-Calder-Mixed}, Theorem 
\ref{Main-Comm-Calder-Mixed}, and Theorem \ref{Main-Comm-Calder-Mixed-Weak}.
\end{proof}

\section{Applications to Partial Differential Equations}

In this section, we investigate the regularity properties of the solution of partial differential equations, i.e. elliptic partial differential equation and parabolic partial differential equation. We use the boundedness properties obtained in the previous sections particularly fractional integral operator $I_\alpha$ and singular integral operators $K$. Here we take $w_1 = w_2 = w$ the weight on $\mathbb{R}^n.$

\medskip

\subsection{Elliptic partial differential equations}

Throughout this subsection, let $\Omega$ an open, bounded, and connected subset of $\mathbb{R}^n$ where $n\geq2$. We assume that $w\in A_{p,q}$ and the pairing functions $(\psi_1,\psi_2)$ satisfy (\ref{inequal-Riesz}). We also assume that $1 \leq p<\infty$ and write $1/q = 1/p - 2/n$ and $1/q_0 = 1/p-1/n.$

For $k=1, 2,$ we denote $W^{1,k}(\Omega)$ by the Sobolev spaces and under the Sobolev norm, 
the closure of $C_0^\infty(\Omega)$ in $W^{1,k}(\Omega)$ is denoted by $W_0^{1,k}(\Omega).$ 
Moreover, we denote $H^{-1}(\Omega)$ by the dual space of $W_0^{1,2}(\Omega).$ We shall 
investigate the following Dirichlet problem
\begin{equation}\label{dirichlet}
\begin{aligned}
	& Lu = f \text{\textit{ in }} \Omega, \quad u = 0 \text{\textit{ in }} 
\partial\Omega, \quad u \in W_o^{1,2} (\Omega)
\end{aligned}
\end{equation}
where $L$ is the divergent elliptic operator. The function $f$ is belonging to generalized 
weighted Morrey spaces $\mathcal{M}^{p,w^p}_{\psi_1}$.
The operator $L$ is defined by
$$
Lu = - \sum_{i, j=1}^\infty \frac{\partial}{\partial x_j} \left(a_{ij} \frac{\partial u}{\partial x_i}\right)
$$
where $u\in W_0^{1,2} (\Omega)$, $a_{ij} = a_{ji} \in L^\infty(\Omega)$ for $i, j, \in \{1, 
\dots, n\}$, and there exists $\Lambda>0$ for which
$$
\Lambda^{-1} |\xi|^2 \leq \sum_{i, j=1}^n a_{ij}(x) \xi_i\xi_j \leq \Lambda |\xi|^2, \quad \xi = (\xi_1, \dots, \xi_n) \in \mathbb{R}^n,
$$
for $x \in \Omega$ a.e. Moreover, we assume that the coefficients of $L$ satisfy the Dini-
continuous condition.

\begin{theorem}\cite{Gruter} \label{Gruter}
There exists a unique function $K:\Omega \times \Omega \to [0,\infty]$ such that
$$
G(\cdot, y) \in W^{1,2}(\Omega \setminus B(y,r)) \cap W_0^{1,1}(\Omega), \quad (y,r) \in 
\Omega \times \mathbb{R}^+,
$$
and for $\phi \in C_0^\infty(\Omega)$,
$$
\int_\Omega \sum_{i,j=1}^n a_{ij}(x) \frac{\partial G(x,y)}{\partial x_i} \frac{\partial (x)}
{\partial x_j} dx = \phi(y).
$$
Furthermore,
$$
G(x,y) \leq C \frac{1}{|x-y|^{n-2}}, \quad |\nabla_x G(x,y)| \leq \frac{1}{|x-y|^{n-1}}, 
\quad x, y \in \Omega, x \neq y.
$$
\end{theorem}

The function $G$ in the theorem is then called the Green function for the operator $L$ and 
the domain $\Omega$. For $f \in \mathcal{M}_\psi^{p,w}(\Omega) \cap H^{-1} (\Omega)$, we 
define
\begin{equation}\label{weaksol}
	u(x) = \int_\Omega G(x,y) f(y) dy, \quad x \in \Omega.
\end{equation}

\begin{lemma} \cite{Tumalun} \label{lemma}
The weak derivative of $u$ is given by
$$
\frac{\partial u(x)}{\partial x_i} = \frac{\partial}{\partial x_i} \left(\int_\Omega G(x,y) 
f(y) dy\right) = \int_\Omega \frac{\partial G(x,y)}{\partial x_i} f(y) dy.
$$
\end{lemma}

\begin{theorem}
For $1<p<\infty,$ there exists a positive constant $C$ such that
$$
\begin{aligned}
	& \|u\|_{\mathcal{M}_{\psi_2}^{q,w^{q}}(\Omega)} \leq C 
 \|f\|_{\mathcal{M}_{\psi_1}^{p,w^p}(\Omega)}, \quad \|u\|_{W\mathcal{M}_{\psi_2}^{q,w^{q}}
 (\Omega)} \leq C \|f\|_{\mathcal{M}_{\psi_1}^{1,w}(\Omega)}.
\end{aligned}
$$
\end{theorem}

\begin{proof}
One may apply Theorem \ref{Main-Riesz} and then use the definition of $u$ as in \ref{weaksol} to get the desired inequality.
\end{proof}

\begin{theorem}
For $1<p<\infty,$ there exists a positive constant $C$ such that
$$
\begin{aligned}
	& \|u\|_{\mathcal{M}_\psi^{p_0,w_0}(0, T\mathcal{M}_{\psi_2}^{q,w^{q}}(\Omega))} 
\leq C \|f\|_{\mathcal{M}_\psi^{p_0,w_0}(0,T,\mathcal{M}_{\psi_1}^{p,w^p}(\Omega))}, \\ 
        & \|u\|_{\mathcal{M}_\psi^{p_0,w_0}(0,T,W\mathcal{M}_{\psi_2}^{q,w^{q}}(\Omega))}   
\leq C \|f\|_{\mathcal{M}_\psi^{p_0,w_0}(0,T,\mathcal{M}_{\psi_1}^{1,w}(\Omega))}.
\end{aligned}
	$$
\end{theorem}

By two previous theorems, Theorem 4 in \cite{Tumalun}, and the results in \cite{DiFazio1993, 
DiFazio2020}, we have the following theorem.

\begin{theorem}
$u$ is the unique weak solution of Dirichlet problem \ref{dirichlet} and 
$$
u \in 
\mathcal{M}_{\psi_2}^{q,w^q}(\Omega) \cap \mathcal{M}_\psi^{p_0,w_0}(0, T, 
\mathcal{M}_{\psi_2}^{q,w^q}(\Omega)).
$$
\end{theorem}

For the gradient of $u$, related to the boundedness on generalized weighted Morrey spaces 
and generalized weighted mixed-Morrey spaces, the following theorem holds.

\begin{theorem}
For $1<p<\infty,$ there exists a constant $C$ such that
$$
\|\nabla u \|_{\mathcal{M}_{\psi_2}^{q_0,w^{q_0}}(\Omega)} \leq C 
\|f\|_{\mathcal{M}_{\psi_1}^{p,w^p}(\Omega)}, \quad \|
\nabla u\|_{W\mathcal{M}_{\psi_2}^{q_0,w^{q_0}}(\Omega)} \leq C 
\|f\|_{\mathcal{M}_{\psi_1}^{1,w}(\Omega)}.
$$
\end{theorem}

\begin{proof}
Note that by Theorem \ref{Gruter} and Lemma \ref{lemma}, we have
$$
\begin{aligned}
|\nabla u|
	= \left(\sum_{i=1}^n \left| \int_\Omega \frac{\partial G(x,y)}{\partial x_i} f(y) 
 dy\right|^2 \right)^\frac12
	\leq \sqrt{n} \int_\Omega |\nabla_x G (x,y)| |f(y)| dy
	\leq \sqrt{n} \int_\Omega\frac{1}{|x-y|^{n-1}} |f(y)| dy,
	\end{aligned}
$$
for $x \in \Omega.$ Hence, by Theorem \ref{Main-Riesz}, we obtain the desired inequality.
\end{proof}

\begin{theorem}
For $1<p<\infty,$ there exists a constant $C$ such that
$$
\begin{aligned}
		&\|\nabla u \|_{\mathcal{M}_\psi^{p_0,w_0}(0,T, 
\mathcal{M}_{\psi_2}^{q_0,w^{q_0}}(\Omega))} \leq C \|f\|_{\mathcal{M}_\psi^{p_0,w_0}
(0,T,\mathcal{M}_{\psi_1}^{p,w^p}(\Omega))}, \\
		& \|\nabla u\|_{\mathcal{M}_\psi^{p_0,w_0}
(0,T,W\mathcal{M}_{\psi_2}^{q_0,w^{q_0}}(\Omega))} \leq C \|f\|_{\mathcal{M}_\psi^{p_0,w_0}
(0,T,\mathcal{M}_{\psi_1}^{1,w}(\Omega))}.
\end{aligned}
$$
\end{theorem}

\medskip

\subsection{Parabolic partial differential equations}

First we recall the Sarason class $VMO = VMO(\mathbb{R}^n).$ The $BMO$ function $f$ is in $VMO$ if
$$
\lim_{r\to 0^+} \sup_{\rho \geq r} \frac1{B_\rho} \int_{B_r} |f(y)-f_{B_\rho}| dy = 0
$$
where $B_\rho$ is a ball with radius $\rho.$ we also recall that the singular integral operator $K$ and its commutator as in Section 5 are bounded on $L^{p,w}$ (see \cite{Guliyev2014, Sarason} for more details).

Through this subsection, we assume that $w=w_1=w_2\in A_p$ where $1 < p<\infty$ and the pairing functions $(\psi_1,\psi_2)$ satisfy (\ref{inequal-comm-calder}). Next, we recall the definition of singular integral operator related to variable Calder\`on-Zymgund kernel.

\begin{defn}\label{Cald-Kern}
Suppose that $K:\mathbb{R}^n\times\mathbb{R}^n\setminus \{0\} \to \mathbb{R}$ is measurable 
function. $K$ is called a variable Calder\`on-Zygmund kernel provided the following 
condition holds.
\begin{enumerate}
	\item $K(x,\cdot)$ is a Calder\`on-Zygmund kernel for almost all $x \in 
 \mathbb{R}^n$, i.e.:
\begin{enumerate}
	\item $K(x,\cdot)\in C^\infty(\mathbb{R}^n\setminus \{0\})$,
	\item $K(x,\mu y) = \mu^{-n} K(x, \mu)$ for $\mu>0$,
	\item $\int_{\mathbb{S}^{n-1}} K(x,y) d\sigma_y = 0$ and $\int_{\mathbb{S}^{n-1}} 
 |K(x,y)| d\sigma_y$.
\end{enumerate}
		\item $\max_{|\beta|\leq 2n} \left\|D^{\beta}_y K\right\|_{L^\infty(\mathbb{R}^n\times \mathbb{S}^{n-1})}$ is finite.
\end{enumerate}
\end{defn}

We define the singular integral operators $K(f)(x)$ and the commutator $[b,K](f)(x)$ by
$$
\begin{aligned}
	K(f)(x) & = P.V. \int_{\mathbb{R}^n} K(x,x-y) f(y) dy\\
	[a,K](f)(x) & = P.V. \int_{\mathbb{R}^n} K(x,x-y) [a(x)-a(y)] f(y) dy.
\end{aligned}
$$
$K$ satisfies \ref{sublinear} for $f \in L_1(\mathbb{R}^n)$ with compact support where 
$\Omega$ constant a.e. as in \cite{Guliyev2019}. Hence, we have the following corollary 
about the boundedness of $K$ and its commutator on generalized weighted Morrey spaces which 
generalize the results in \cite{Guliyev2019, Ragusa2017, Ramadana, Wang} and on on 
generalized weighted mixed-Morrey spaces which generalize the results in \cite{Ragusa2017}.

\begin{cor}\label{Cald-Cor-Kern}
Let $\psi$ be a positive function on $\mathbb{R}^n \times \mathbb{R}^+$ and $1<p_0<\infty$. 
Then, $K$ and $[b, K]$ are bounded from $\mathcal{M}_{\psi_1}^{p,w}$ to 
$\mathcal{M}_{\psi_2}^{p,w}$ and $\mathcal{M}_\psi^{p_0,w_0}(0, T, \mathcal{M}^{p,w}_{\psi_1})$ to $\mathcal{M}_\psi^{p_0,w_0}(0, T, 
\mathcal{M}^{p,w}_{\psi_2})$.
\end{cor}

Let $n \geq 3$ and $Q_T = \Omega\times (0,T)$ be a cylinder of $\mathbb{R}^{n+1}$ of base 
$\Omega \subset \mathbb{R}^n.$ In the sequel we write $x'=(x,t) = (x_1, x_2, \cdots, x_n, 
t)\in \mathbb{R}^{n+1}$ a generic point in $Q_T$, $f\in \mathcal{M}_{\psi_1}^{p,w}(\Omega)$ 
or $f \in \mathcal{M}_{\psi}^{p_0,w_0}(0,T, \mathcal{M}_{\psi_1}^{p,w}(\Omega))$ where 
$1\leq p_0,p<\infty$.

We also write $$Lu = u_t - \sum_{i,j=1}^n a_{ij}(x,t) \frac{\partial^2 u}{\partial x_i \partial x_j} $$
where
$$
\begin{aligned}
	& a_{ij}(x,t) = a_{ji}(x,t); \quad i,j = 1, \cdots, n, \quad \text{\textit{a.e. }}x \in Q_T, \\
	& \Lambda^{-1} |\xi|^2 \leq \sum_{i,j=1}^n a_{ij}(x,t) \xi_i\xi_j \leq \Lambda |\xi|^2; \quad \text{\textit{a.e. in }}Q_T, \xi \in \mathbb{R}^n,\\
	& a_{ij}(x,t) \in VMO(Q_T)\cap L^\infty (Q_T); \quad i,j=1, \cdots, n.
\end{aligned}
$$

Suppose the equation $Lu(x,t)=f(x,t)$. The strong solution of the equation is a function 
$u \in \mathcal{M}_{\psi_1}^{p,w}(\Omega)$ if $f \in u \in \mathcal{M}_{\psi_1}^{p,w}
(\Omega)$ and $u \in \mathcal{M}_\psi^{p_0,w_0}(0,T,\mathcal{M}_{\psi_1}^{p,w}(\Omega))$ if 
$f \in u \in \mathcal{M}_\psi^{p_0,w_0}(0,T,\mathcal{M}_{\psi_1}^{p,w}(\Omega))$ with all 
its weak derivatives $D_{x'_i}u,D_{x'_ix'_j}u$, and $D_t$ satisfying the the equation as in 
\cite{Ragusa2017}.  The local representation formula (see \cite{Bramanti} for more details 
about the formula and notations) for $D_{x'_i,x'_j} u$ where $u\in C_t = \{v\in 
C_0^\infty(\mathbb{R}^{n+1}\cap \{t \geq 0\}) : v(x,0) = 0\}$ is given by

\begin{equation} \label{Local-Repr}
\begin{aligned}
D_{x'_i,x'_j} & = \lim_{\varepsilon \to 0} \int_{\rho(x-y)>\varepsilon} \Gamma_{ij}(x,x-y) 
Lu(y) dy \\
&+ \lim_{\varepsilon \to 0} \int_{\rho(x-y)>\varepsilon} \Gamma_{ij}(x,x-y) \sum_{h,k=1}^n 
 [a_{hk}(y)-a_{hk}(x)] D_{x'_h,x'_k}u(y) dy + Lu(x) \int_\Sigma \Gamma_j(x,y) v_i 
 d\sigma (y) \\
& = \mathcal{R}_{ij}Lu(x) + \sum_{h,k=1}^n [a_{hk}, \mathcal{R}_{ij}] (D_{x'_h,x'_k}u)(x) + 
Lu(x) \int_{\Sigma} \Gamma_j(x,y) v_i d\sigma(y),
\end{aligned}
\end{equation}
for $x \in {\rm supp}\,(f)$ where $v_i$ is the $i-$th component of the outer normal to 
$\Sigma$. Here $\Gamma_{ij} (x,\xi) = D_{\xi_i, \xi_j} \Gamma(x,\xi)$ and $\Gamma_{ij}$ are 
Calder\`on-Zygmund kernel as in Definition \ref{Cald-Kern}. This implies that 
$\mathcal{R}_{ij}$ are singular integrals as in section 5. Moreover, we may observe that
$$
u_t = Lu + \sum_{i,j=1}^n a_{ij}(x,t) \frac{\partial^2 u}{\partial x_i \partial x_j}.
$$

Therefore, by Corollary \ref{Cald-Cor-Kern}, the fact that $K$ and $[b,K]$ is bounded on 
$L^{p,w}$ provided $b \in VMO$, and the local representation \ref{Local-Repr} we have the 
following theorems about the regularity properties of the solution of the parabolic partial 
differential equations.

\begin{theorem}
Let $n\geq 3, a_{ij}\in VMO(Q_T) \cap L^\infty (Q_T)$, $B_r \subseteq\Omega$ a ball in 
$\mathbb{R}^n.$ For every $u$ with compact support in $B \times (0,T)$, the solution of 
$Lu=f$ such that $D_{x'_i,x'_j}u \in \mathcal{M}_{\psi_1}^{p,w}(B_r)$ for $i,j=1,\cdots, n$, 
there exists $r_0$ such that if $r<r_0$, then
$$
\begin{aligned}
	& \|D_{x'_i, x'_j} u\|_{\mathcal{M}_{\psi_2}^{p,w}(B_r)} \leq 
 \|Lu\|_{\mathcal{M}_{\psi_1}^{p,w}(B_r)}, \quad i,j =1, \cdots, n,\\
	& \|u_t\|_{\mathcal{M}_{\psi_2}^{p,w}(B_r)} \leq \|Lu\|_{\mathcal{M}_{\psi_1}^{p,
 w}(B_r)}.
\end{aligned}
$$
\end{theorem}

\begin{theorem}
Let $n\geq 3, a_{ij}\in VMO(Q_T) \cap L^\infty (Q_T)$, and $B_r \subseteq\Omega$ a ball in 
$\mathbb{R}^n.$ For every $u$ with compact support in $B \times (0,T)$, the solution of 
$Lu=f$ such that $D_{x'_i,x'_j}u \in \mathcal{M}_\psi^{p_0}(0,T\mathcal{M}_{\psi_1}^{p,w}
(B_r))$ for $i,j=1,\cdots, n$, there exists $r_0$ such that if $r<r_0$, then
$$
\begin{aligned}
		& \|D_{x'_i, x'_j} u\|_{\mathcal{M}_\psi^{p_0}(0,T\mathcal{M}_{\psi_2}^{p,w}(B_r))} \leq \|Lu\|_{\mathcal{M}_\psi^{p_0}(0,T\mathcal{M}_{\psi_1}^{p,w}(B_r))}, \quad i,j =1, \cdots, n,\\
		& \|u_t\|_{\mathcal{M}_\psi^{p_0}(0,T\mathcal{M}_{\psi_2}^{p,w}(B_r))} \leq \|Lu\|_{\mathcal{M}_\psi^{p_0}(0,T\mathcal{M}_{\psi_1}^{p,w}(B_r))}.
	\end{aligned}
$$
\end{theorem}

\begin{remark}
    Note that the last theorem generalizes Theorem 5.1 in \cite{Ragusa2017}.
\end{remark}

\bibliographystyle{amsplain}

\end{document}